\documentclass[11pt]{article}

\usepackage[a4paper,top=3cm,bottom=3.5cm,left=2.5cm,right=2.5cm,marginparwidth=1.75cm]{geometry}
\usepackage[english]{babel}
\usepackage[utf8x]{inputenc}
\usepackage[T1]{fontenc}
\usepackage{graphicx}
\usepackage{caption, subcaption}
\usepackage[title]{appendix}
\usepackage{tikz}
\usepackage{pgfplots}

\usepackage{cutwin}
\graphicspath{ {images/} }
\usetikzlibrary{intersections,decorations.pathreplacing,decorations.markings,calc,angles,quotes,arrows.meta,pgfplots.fillbetween,patterns}

\usepackage{amsmath,yhmath}
\usepackage{amsthm}
\usepackage{amssymb}
\usepackage{hyperref}
\usepackage{cleveref}
\usepackage{epstopdf}
\usepackage{comment}
\usepackage{todonotes}
\usepackage{color}
\usepackage{float}
\usepackage[sort&compress,numbers]{natbib}
\pgfplotsset{compat=newest}
\usetikzlibrary{shapes.geometric}
\usetikzlibrary{arrows,automata,quotes}
\usepackage{iftex}
\usepackage{xfrac} 
\usepackage{amsmath,yhmath}
\usepackage{amsthm}
\usepackage{amssymb}
\usepackage{mathtools}
\usepackage{nccmath}
\usepackage{hyperref}
\usepackage{caption, subcaption}
\usepackage[title]{appendix}

\usepackage{cleveref}
\usepackage{epstopdf}
\usepackage{comment}
\usepackage{todonotes}
\usepackage{color}
\usepackage{float}
\usepackage{cutwin}
\usepackage[sort&compress,numbers]{natbib}
\usepackage{makecell} 
\usepackage{array}   
\newcolumntype{C}{>{$}c<{$}} 
\usepackage{multirow}
\newcommand\wline[1]{\noalign{\hrule height #1}}

\usepackage{graphicx}
\usepackage{caption}
\usepackage{subcaption}
\graphicspath{ {./my/} }
\usetikzlibrary{intersections,decorations.pathreplacing,decorations.markings,calc,angles,quotes,arrows.meta,pgfplots.fillbetween,patterns}
\usepackage[section]{placeins}

\newcommand{\al}{\alpha}
\newcommand{\w}{\omega}
\newcommand{\eps}{\varepsilon}
\newcommand{\kk}{\kappa}
\newcommand{\vp}{\varphi}
\newcommand{\Z}{\mathbb{Z}}
\newcommand{\N}{\mathbb{N}}

\newcommand{\R}{\mathbb{R}}
\newcommand{\C}{\mathbb{C}}

\newcommand{\LL}{\mathcal{L}}

\newcommand{\FF}{\mathcal{F}}

\newcommand{\p}{\partial}
\newcommand{\dx}{{d}}

\newcommand{\nm}{\noalign{\smallskip}}
\newcommand{\ds}{\displaystyle}
\newcommand{\iu}{\mathrm{i}\mkern1mu}

\newtheorem{thm}{Theorem}[section]
\newtheorem{cor}[thm]{Corollary}
\newtheorem{lemma}[thm]{Lemma}
\newtheorem{prop}[thm]{Proposition}

\theoremstyle{definition}
\newtheorem{definition}{Definition}[section]

\newtheorem{rem}[definition]{Remark}

\title{Transmission properties of space-time modulated metamaterials\thanks{\footnotesize
This work was supported in part by the Swiss National Science Foundation grant number
200021--200307.}}
\author{Habib Ammari\thanks{\footnotesize Department of Mathematics, 
		ETH Z\"urich, 
		R\"amistrasse 101, CH-8092 Z\"urich, Switzerland (habib.ammari@math.ethz.ch, jinghao.cao@sam.math.ethz.ch, xinzeng@student.ethz.ch).} \and Jinghao Cao\footnotemark[2] \and Xinmeng Zeng\footnotemark[2] }
\date{}
\begin{document}
\maketitle
\begin{abstract}
	We prove the possibility of achieving exponentially growing wave propagation in space-time modulated media and give an asymptotic analysis of the quasifrequencies in terms of the amplitude of the time modulation at the degenerate points of the folded band structure. Our analysis provides the first proof of existence of k-gaps in the band structures of space-time modulated systems of subwavelength resonators. 
	\end{abstract}
\noindent{\textbf{Mathematics Subject Classification (MSC2000):} 35J05, 35C20, 35P20, 74J20}
\vspace{0.2cm}
	\vspace{0.5cm}	
\section{Introduction}
The ability to control wave propagation is fundamental in many areas of physics. In particular, systems with subwavelength structures have attracted considerable attention over the past decades \cite{lemoult2016soda,yves2017crytalline,
phononic1,phononic2}. The word \textit{subwavelength} designates systems that are able to strongly scatter waves with comparatively large wavelengths. The building blocks of such systems which exhibit subwavelength resonance are called \textit{subwavelength resonators}. Arranged in repeating patterns inside a medium with highly different material parameters, the subwavelength resonators together with the background medium can form microstructures that exhibit various new properties that the base materials do not possess. Such structures are examples of \textit{metamaterials}: materials with repeating microstructures that possess additional properties derived from the newly designed structures. These properties typically depend on the shape, geometry, and arrangement of the unit cells. Artificially structured metamaterials have applications in e.g. optics, nanophotonics and acoustics, and have been studied in detail for their subwavelength characteristics \cite{ammari2021functional,review,review2, metamaterials2013}.

As reviewed in \cite{ammari2021functional}, \emph{high-contrast} resonators are a natural choice of resonators when designing subwavelength metamaterials. In fact, the high contrast material parameters are crucial to their subwavelength properties. Such structures can be used to achieve a variety of effects. These include superresolution and superfocusing, double negative material properties and robust guiding properties at subwavelength scales \cite{davies2019fully, ammari2018minnaert,ammari2020exceptional,ammari2020highfrequency,ammari2020highorder,ammari2017subwavelength,ammari2017double,ammari2020honeycomb,ammari2020subwavelength,ammari2020robust,ammari2021bound,ammari2020topological}. Inspiration for such resonators came from the resonance of air bubbles in water, as observed by Marcel Minnaert in 1933 \cite{minnaert1933musical}. They are known to possess a subwavelength resonance called the Minnaert resonance \cite{ammari2018minnaert}. Other examples of subwavelength resonators are Helmholtz resonators, plasmonic nanoparticles, and high-dielectric nanoparticles; see, for instance, \cite{ammari2015superresolution,ammari2017plasmonicscalar,
ammari2016plasmonicMaxwell,ammari2020maxwell,pierre1,pierre2,bryn,hyeonbae,hongyu,john}.

The static case with periodic spatial systems of subwavelength resonators has been studied with the help of Floquet-Bloch theory. In such structures, wave momentums are defined modulo the dual lattice and are contained within the Brillouin zone. The so-called \textit{band structure} of the material describes the frequency-to-momentum relationship of waves inside the material. It has been shown in \cite{ammari2017subwavelength} that there exists a subwavelength band gap, i.e., a gap between the band functions where waves with subwavelength frequencies inside the band gap cannot propagate through the material and will decay exponentially.

Inspired by the periodic space structure of the system of resonators, it is natural to consider modulations of the material parameters that are periodic in time. Many intriguing phenomena arise from such modulations \cite{koutserimpas2020electromagnetic, nassar2018quantization}. For instance, time modulation provides a way to break reciprocity, leading to non-reciprocal transmission properties \cite{reviewTM1,reviewTM2,reviewTM3}. The mathematical foundation for studying such systems has been recently developed in \cite{ammari2020time, jinghaothesis,TheaThesis,paper2}.
In particular, periodic lattices of acoustic subwavelength resonators have been studied in \cite{jinghaothesis}. It has been shown that temporally modulating the density parameter breaks the time-reversal symmetry and leads to unidirectional excitation of the operating waves. In fact, time modulation of the density parameter turns degenerate points of the folded band structure into non-symmetric band gaps.
The degenerate points are obtained by folding of the band functions of the unmodulated structure.

In this paper, we study time dependent periodic systems in the context of high-contrast acoustic subwavelength metamaterials. Namely, we are interested in the case where time modulation is applied uniformly inside each resonator with a possible phase shift. We consider both time modulations in the density and  the bulk modulus of the acoustic subwavelength resonators. Using the capacitance matrix formulation, the initial wave equation can be viewed, up to first-order in the density contrast parameter, as a system of Hill's equations. We study the solutions to this system  using methods similar to perturbation theory in quantum mechanics, and we investigate the cases where one or both of the density  and the bulk modulus are modulated. Through an asymptotic analysis of the quasifrequencies in terms of the amplitude of the modulations that are obtained by perturbing degenerate points of the folded unmodulated band structure, we show that non-reciprocal transmission occurs only in the case where the density alone is modulated. Modulating the bulk modulus yields exponentially growing waves inside the structure, since the quasifrequencies at the degenerate points are purely imaginary. Around these degenerate points,  it can be shown that by time modulations the band functions open into band gaps in the momentum variable, known as \emph{$k$-gaps}  \cite{ammari2020time,koutserimpas2018electromagnetic,cassedy1967dispersion}. 
To the best of our knowledge, our analysis in this paper provides the first  
proof of existence of k-gaps in the band structures of space-time modulated systems of subwavelength resonators. 

This paper is organized as follows. In section \ref{sec00}, the problem is formulated mathematically. Moreover, we briefly review Floquet-Bloch theory which is an essential tool for solving ordinary differential equations (ODEs) with periodic coefficients. We also introduce the quasiperiodic capacitance matrix which provides a discrete approximation for computing the band functions of periodic systems of subwavelength resonators and recall the notion of subwavelength quasifrequencies. These are quasifrequencies with associated Bloch modes essentially supported in the low-frequency regime. They can be approximated as the quasifrequencies associated with solutions to a system of Hill's equations. 
In section \ref{sec2}, we first reduce such a system of Hill's equations to a first-order ODE and recall the Floquet decomposition of its fundamental solution.  
Then we compute the first-order terms in the asymptotic expansions of the quasifrequencies at degenerate points in terms of the modulation amplitude for a one-dimensional lattice and compare the cases where one or both material parameters are time modulated. Our results in this section are illustrated by a variety of numerical simulations. Finally, section \ref{sec4} is devoted to numerical illustrations of k-gap openings in  square and  honeycomb lattices. In the appendix, we recall a basic result in eigenvalue perturbation theory, which is reformulated to suit our setting. 

\section{Problem formulation and preliminary theory} \label{sec00}
\subsection{Resonator structure and wave equation}
\label{sec0}
Consider the wave equation in time modulated media. Such wave equations can be used to model not only acoustic waves but also polarized electromagnetic waves.
The time dependent material parameters are given by $\rho(x,t)$ and $\kappa(x,t)$. In acoustics, $\rho$ and $\kappa$ represent the density and the bulk modulus of the material.

We will study the time dependent wave equation in dimension $d=2$ or $3$:
\begin{equation}
\label{waveequation}
\left(\frac{\partial}{\partial t }\frac{1}{\kappa(x,t)}\frac{\partial}{\partial t}-\nabla\cdot\frac{1}{\rho(x,t)}\nabla\right)u(x,t)=0,\ \ x\in \mathbb{R}^d,t\in\mathbb{R}.	
\end{equation}
We assume that the metamaterial is periodic with lattice $\Lambda \subset \R^d$ and unit cell $Y\subset \R^d$. Each unit cell contains a system of $N$ resonators $D\Subset Y$. $D$ is constituted by $N$ disjoint domains $D_i$. For $i=1,\ldots,N$, each $D_i$ is connected and has a boundary of Hölder class: $\p D_i \in C^{1,s}, 0 < s < 1$. Let $\mathcal{C}_i$ denote the periodically repeated $i^{\text{th}}$ resonators and $\mathcal{C}$ the full crystal:
\begin{equation}
	\mathcal{C}_i=\bigcup_{m\in \Lambda}D_i+m,\ \ \mathcal{C}=\bigcup_{m\in\Lambda}D+m.
\end{equation} 	
Let $\Lambda^*$ be the dual lattice and define the (space-) Brillouin zone $Y^*$ as the torus $Y^*:= \mathbb{R}^d/\Lambda^*$. Recall that if $\Lambda\subset \R^d$ is the lattice generated by the vectors $l_1,\ldots,l_d$, then the dual lattice $\Lambda^*$ is generated by the vectors $2\pi l_i^*$, where $l_i^*$ is the dual vector of $l_i$ with $l_i\cdot l_j^*=\delta_{ij}$; see Figure \ref{fig:1}.

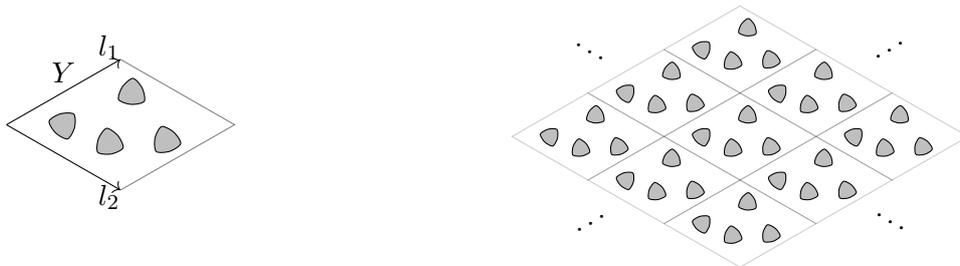
\begin{figure}[H]
		\begin{subfigure}[b]{0.4\linewidth}
			\centering
			\begin{tikzpicture}[scale=1.5]
				\begin{scope}[scale=1]
					\pgfmathsetmacro{\rb}{0.13pt}
					\pgfmathsetmacro{\rs}{0.1pt}
					\pgfmathsetmacro{\ao}{326}
					\pgfmathsetmacro{\at}{46}
					\pgfmathsetmacro{\ad}{228}
					\pgfmathsetmacro{\af}{100}
					\coordinate (a) at (1,{1/sqrt(3)});		
					\coordinate (b) at (1,{-1/sqrt(3)});	
					\coordinate (c) at (2,0);
					\draw[->] (0,0) -- (a) node[pos=0.9,xshift=0,yshift=7]{ $l_1$} node[pos=0.5,above]{$Y$};
					\draw[->] (0,0) -- (b) node[pos=0.9,xshift=0,yshift=-5]{ $l_2$};
					\draw[opacity=0.5] (a) -- (c) -- (b);
					\begin{scope}[xshift = 1.1cm, yshift=0.28cm,rotate=\ao]
						\draw[fill=lightgray] plot [smooth cycle] coordinates {(0:\rb) (60:\rs) (120:\rb) (180:\rs) (240:\rb) (300:\rs) };
					\end{scope}
					\begin{scope}[xshift = 0.5cm, rotate=\at]
						\draw[fill=lightgray] plot [smooth cycle] coordinates {(0:\rb) (60:\rs) (120:\rb) (180:\rs) (240:\rb) (300:\rs) };
					\end{scope}
					\begin{scope}[xshift = 1.4cm, yshift=-0.14cm, rotate=\ad]
						\draw[fill=lightgray] plot [smooth cycle] coordinates {(0:\rb) (60:\rs) (120:\rb) (180:\rs) (240:\rb) (300:\rs) };
					\end{scope}
					\begin{scope}[xshift = 0.9cm,yshift=-0.16cm, rotate=\af]
						\draw[fill=lightgray] plot [smooth cycle] coordinates {(0:\rb) (60:\rs) (120:\rb) (180:\rs) (240:\rb) (300:\rs) };
					\end{scope}
				\end{scope}
			\end{tikzpicture}
			\vspace{0.65cm}
			\caption{Unit cell $Y$ containing $N=4$ resonators.}
		\end{subfigure}
		\begin{subfigure}[b]{0.6\linewidth}
		\centering
			\begin{tikzpicture}[scale=1]
				\begin{scope}[xshift=-5cm,scale=1]
					\coordinate (a) at (1,{1/sqrt(3)});		
					\coordinate (b) at (1,{-1/sqrt(3)});	
					\coordinate (c) at (2,0);
					\pgfmathsetmacro{\rb}{0.13pt}
					\pgfmathsetmacro{\rs}{0.1pt}
					\pgfmathsetmacro{\ao}{326}
					\pgfmathsetmacro{\at}{46}
					\pgfmathsetmacro{\ad}{228}
					\pgfmathsetmacro{\af}{100}	
					\draw[opacity=0.2] (0,0) -- (a);
					\draw[opacity=0.2] (0,0) -- (b);
					\draw[opacity=0.2] (a) -- (c) -- (b);
					\begin{scope}[xshift = 1.1cm, yshift=0.28cm,rotate=\ao]
						\draw[fill=lightgray] plot [smooth cycle] coordinates {(0:\rb) (60:\rs) (120:\rb) (180:\rs) (240:\rb) (300:\rs) };
				\end{scope}
					
				\begin{scope}[xshift = 0.5cm, rotate=\at]
				
						\draw[fill=lightgray] plot [smooth cycle] coordinates {(0:\rb) (60:\rs) (120:\rb) (180:\rs) (240:\rb) (300:\rs) };
					\end{scope}
					\begin{scope}[xshift = 1.4cm, yshift=-0.14cm, rotate=\ad]
						\draw[fill=lightgray] plot [smooth cycle] coordinates {(0:\rb) (60:\rs) (120:\rb) (180:\rs) (240:\rb) (300:\rs) };
					\end{scope}
					\begin{scope}[xshift = 0.9cm,yshift=-0.16cm, rotate=\af]
						\draw[fill=lightgray] plot [smooth cycle] coordinates {(0:\rb) (60:\rs) (120:\rb) (180:\rs) (240:\rb) (300:\rs) };
					\end{scope}			
					\begin{scope}[shift = (a)]
						\draw[opacity = 0.2] (0,0) -- (1,{1/sqrt(3)}) -- (2,0) -- (1,{-1/sqrt(3)}) -- cycle; 
						\begin{scope}[xshift = 1.1cm, yshift=0.28cm,rotate=\ao]
							\draw[fill=lightgray] plot [smooth cycle] coordinates {(0:\rb) (60:\rs) (120:\rb) (180:\rs) (240:\rb) (300:\rs) };
						\end{scope}
						\begin{scope}[xshift = 0.5cm, rotate=\at]
							\draw[fill=lightgray] plot [smooth cycle] coordinates {(0:\rb) (60:\rs) (120:\rb) (180:\rs) (240:\rb) (300:\rs) };
						\end{scope}
						\begin{scope}[xshift = 1.4cm, yshift=-0.14cm, rotate=\ad]
							\draw[fill=lightgray] plot [smooth cycle] coordinates {(0:\rb) (60:\rs) (120:\rb) (180:\rs) (240:\rb) (300:\rs) };
						\end{scope}
						\begin{scope}[xshift = 0.9cm,yshift=-0.16cm, rotate=\af]
							\draw[fill=lightgray] plot [smooth cycle] coordinates {(0:\rb) (60:\rs) (120:\rb) (180:\rs) (240:\rb) (300:\rs) };
						\end{scope}
					\end{scope}
					\begin{scope}[shift = (b)]
						\draw[opacity = 0.2] (0,0) -- (1,{1/sqrt(3)}) -- (2,0) -- (1,{-1/sqrt(3)}) -- cycle; 
						\begin{scope}[xshift = 1.1cm, yshift=0.28cm,rotate=\ao]
							\draw[fill=lightgray] plot [smooth cycle] coordinates {(0:\rb) (60:\rs) (120:\rb) (180:\rs) (240:\rb) (300:\rs) };
						\end{scope}
						\begin{scope}[xshift = 0.5cm, rotate=\at]
							\draw[fill=lightgray] plot [smooth cycle] coordinates {(0:\rb) (60:\rs) (120:\rb) (180:\rs) (240:\rb) (300:\rs) };
						\end{scope}
						\begin{scope}[xshift = 1.4cm, yshift=-0.14cm, rotate=\ad]
							\draw[fill=lightgray] plot [smooth cycle] coordinates {(0:\rb) (60:\rs) (120:\rb) (180:\rs) (240:\rb) (300:\rs) };
						\end{scope}
						\begin{scope}[xshift = 0.9cm,yshift=-0.16cm, rotate=\af]
						\draw[fill=lightgray] plot [smooth cycle] coordinates {(0:\rb) (60:\rs) (120:\rb) (180:\rs) (240:\rb) (300:\rs) };
						\end{scope}
					\end{scope}
					\begin{scope}[shift = ($-1*(a)$)]
						\draw[opacity = 0.2] (0,0) -- (1,{1/sqrt(3)}) -- (2,0) -- (1,{-1/sqrt(3)}) -- cycle; 
						\begin{scope}[xshift = 1.1cm, yshift=0.28cm,rotate=\ao]
						\draw[fill=lightgray] plot [smooth cycle] coordinates {(0:\rb) (60:\rs) (120:\rb) (180:\rs) (240:\rb) (300:\rs) };
						\end{scope}
						\begin{scope}[xshift = 0.5cm, rotate=\at]
						\draw[fill=lightgray] plot [smooth cycle] coordinates {(0:\rb) (60:\rs) (120:\rb) (180:\rs) (240:\rb) (300:\rs) };
						\end{scope}
						\begin{scope}[xshift = 1.4cm, yshift=-0.14cm, rotate=\ad]
						\draw[fill=lightgray] plot [smooth cycle] coordinates {(0:\rb) (60:\rs) (120:\rb) (180:\rs) (240:\rb) (300:\rs) };
						\end{scope}
						\begin{scope}[xshift = 0.9cm,yshift=-0.16cm, rotate=\af]
						\draw[fill=lightgray] plot [smooth cycle] coordinates {(0:\rb) (60:\rs) (120:\rb) (180:\rs) (240:\rb) (300:\rs) };
						\end{scope}
					\end{scope}
					\begin{scope}[shift = ($-1*(b)$)]
						\draw[opacity = 0.2] (0,0) -- (1,{1/sqrt(3)}) -- (2,0) -- (1,{-1/sqrt(3)}) -- cycle; 
						\begin{scope}[xshift = 1.1cm, yshift=0.28cm,rotate=\ao]
						\draw[fill=lightgray] plot [smooth cycle] coordinates {(0:\rb) (60:\rs) (120:\rb) (180:\rs) (240:\rb) (300:\rs) };
						\end{scope}
						\begin{scope}[xshift = 0.5cm, rotate=\at]
						\draw[fill=lightgray] plot [smooth cycle] coordinates {(0:\rb) (60:\rs) (120:\rb) (180:\rs) (240:\rb) (300:\rs) };
						\end{scope}
						\begin{scope}[xshift = 1.4cm, yshift=-0.14cm, rotate=\ad]
						\draw[fill=lightgray] plot [smooth cycle] coordinates {(0:\rb) (60:\rs) (120:\rb) (180:\rs) (240:\rb) (300:\rs) };
						\end{scope}
						\begin{scope}[xshift = 0.9cm,yshift=-0.16cm, rotate=\af]
						\draw[fill=lightgray] plot [smooth cycle] coordinates {(0:\rb) (60:\rs) (120:\rb) (180:\rs) (240:\rb) (300:\rs) };
						\end{scope}
					\end{scope}
					\begin{scope}[shift = ($(a)+(b)$)]
						\draw[opacity = 0.2] (0,0) -- (1,{1/sqrt(3)}) -- (2,0) -- (1,{-1/sqrt(3)}) -- cycle; 
						\begin{scope}[xshift = 1.1cm, yshift=0.28cm,rotate=\ao]
						\draw[fill=lightgray] plot [smooth cycle] coordinates {(0:\rb) (60:\rs) (120:\rb) (180:\rs) (240:\rb) (300:\rs) };
						\end{scope}
						\begin{scope}[xshift = 0.5cm, rotate=\at]
						\draw[fill=lightgray] plot [smooth cycle] coordinates {(0:\rb) (60:\rs) (120:\rb) (180:\rs) (240:\rb) (300:\rs) };
						\end{scope}
						\begin{scope}[xshift = 1.4cm, yshift=-0.14cm, rotate=\ad]
							\draw[fill=lightgray] plot [smooth cycle] coordinates {(0:\rb) (60:\rs) (120:\rb) (180:\rs) (240:\rb) (300:\rs) };
						\end{scope}
						\begin{scope}[xshift = 0.9cm,yshift=-0.16cm, rotate=\af]
							\draw[fill=lightgray] plot [smooth cycle] coordinates {(0:\rb) (60:\rs) (120:\rb) (180:\rs) (240:\rb) (300:\rs) };
						\end{scope}
					\end{scope}
					\begin{scope}[shift = ($-1*(a)-(b)$)]
						\draw[opacity = 0.2] (0,0) -- (1,{1/sqrt(3)}) -- (2,0) -- (1,{-1/sqrt(3)}) -- cycle; 
						\begin{scope}[xshift = 1.1cm, yshift=0.28cm,rotate=\ao]
							\draw[fill=lightgray] plot [smooth cycle] coordinates {(0:\rb) (60:\rs) (120:\rb) (180:\rs) (240:\rb) (300:\rs) };
						\end{scope}
						\begin{scope}[xshift = 0.5cm, rotate=\at]
							\draw[fill=lightgray] plot [smooth cycle] coordinates {(0:\rb) (60:\rs) (120:\rb) (180:\rs) (240:\rb) (300:\rs) };
						\end{scope}
						\begin{scope}[xshift = 1.4cm, yshift=-0.14cm, rotate=\ad]
							\draw[fill=lightgray] plot [smooth cycle] coordinates {(0:\rb) (60:\rs) (120:\rb) (180:\rs) (240:\rb) (300:\rs) };
						\end{scope}
						\begin{scope}[xshift = 0.9cm,yshift=-0.16cm, rotate=\af]
							\draw[fill=lightgray] plot [smooth cycle] coordinates {(0:\rb) (60:\rs) (120:\rb) (180:\rs) (240:\rb) (300:\rs) };
						\end{scope}
					\end{scope}
					\begin{scope}[shift = ($(a)-(b)$)]
						\draw[opacity = 0.2] (0,0) -- (1,{1/sqrt(3)}) -- (2,0) -- (1,{-1/sqrt(3)}) -- cycle; 
						\begin{scope}[xshift = 1.1cm, yshift=0.28cm,rotate=\ao]
							\draw[fill=lightgray] plot [smooth cycle] coordinates {(0:\rb) (60:\rs) (120:\rb) (180:\rs) (240:\rb) (300:\rs) };
						\end{scope}
						\begin{scope}[xshift = 0.5cm, rotate=\at]
							\draw[fill=lightgray] plot [smooth cycle] coordinates {(0:\rb) (60:\rs) (120:\rb) (180:\rs) (240:\rb) (300:\rs) };
						\end{scope}
						\begin{scope}[xshift = 1.4cm, yshift=-0.14cm, rotate=\ad]
							\draw[fill=lightgray] plot [smooth cycle] coordinates {(0:\rb) (60:\rs) (120:\rb) (180:\rs) (240:\rb) (300:\rs) };
						\end{scope}
						\begin{scope}[xshift = 0.9cm,yshift=-0.16cm, rotate=\af]
							\draw[fill=lightgray] plot [smooth cycle] coordinates {(0:\rb) (60:\rs) (120:\rb) (180:\rs) (240:\rb) (300:\rs) };
						\end{scope}
					\end{scope}
					\begin{scope}[shift = ($-1*(a)+(b)$)]
						\draw[opacity = 0.2] (0,0) -- (1,{1/sqrt(3)}) -- (2,0) -- (1,{-1/sqrt(3)}) -- cycle; 
						\begin{scope}[xshift = 1.1cm, yshift=0.28cm,rotate=\ao]
							\draw[fill=lightgray] plot [smooth cycle] coordinates {(0:\rb) (60:\rs) (120:\rb) (180:\rs) (240:\rb) (300:\rs) };
						\end{scope}
						\begin{scope}[xshift = 0.5cm, rotate=\at]
							\draw[fill=lightgray] plot [smooth cycle] coordinates {(0:\rb) (60:\rs) (120:\rb) (180:\rs) (240:\rb) (300:\rs) };
						\end{scope}
						\begin{scope}[xshift = 1.4cm, yshift=-0.14cm, rotate=\ad]
							\draw[fill=lightgray] plot [smooth cycle] coordinates {(0:\rb) (60:\rs) (120:\rb) (180:\rs) (240:\rb) (300:\rs) };
						\end{scope}
						\begin{scope}[xshift = 0.9cm,yshift=-0.16cm, rotate=\af]
							\draw[fill=lightgray] plot [smooth cycle] coordinates {(0:\rb) (60:\rs) (120:\rb) (180:\rs) (240:\rb) (300:\rs) };
						\end{scope}
					\end{scope}
					\begin{scope}[shift = ($2*(a)$)]
						\draw (1,0) node[rotate=30]{$\cdots$};
					\end{scope}
					\begin{scope}[shift = ($-2*(a)$)]
						\draw (1,0) node[rotate=210]{$\cdots$};
					\end{scope}
					\begin{scope}[shift = ($2*(b)$)]
						\draw (1,0) node[rotate=-30]{$\cdots$};
					\end{scope}
					\begin{scope}[shift = ($-2*(b)$)]
						\draw (1,0) node[rotate=150]{$\cdots$};
					\end{scope}
				\end{scope}
			\end{tikzpicture}
			\caption{Infinite, periodic system with unit cell $Y$ and 
			lattice $\Lambda$.}
			\label{fig:1}
		\end{subfigure}
		\caption{Illustrations of the unit cell and the periodic system of resonators.} 
\end{figure}
\FloatBarrier

\subsection{Floquet-Bloch theory}
\label{sec1}
\label{Floquettheory}
Floquet theory is widely used for solving differential equations with periodic coefficients. Here, we are interested in solving differential equations of the form \begin{equation}
\label{ode}
\frac{dx}{dt}=A(t)x,\end{equation} 
where $A(t):[0,\infty)\to\C$ is a $T$-periodic $n\times n$ complex matrix function of class $\mathcal{C}^0$.
Recall that the fundamental matrix of (\ref{ode}) is a $n\times n$ matrix with linear independent column vectors, which are solutions to (\ref{ode}). Then the following classical theorem gives us a fundamental result central to our analysis.

\begin{thm}[Floquet's theorem]
	 Let $A(t)$ be a $T$-periodic $n\times n$ complex matrix. Denote by $X(t)$ the fundamental matrix to (\ref{ode}): 
    \begin{equation}
	\label{odex}
	\begin{cases}
		\frac{dX}{dt} =A(t)X,\\
		X(0)=\mathrm{Id}_n,
	\end{cases}
	\end{equation}
	where $\mathrm{Id}_n$ denotes the $n\times n$ identity matrix.
	Then there exists a constant matrix $F$ and a $T$-periodic matrix function $P(t)$ such that
	\begin{equation}
	\label{floquetthm}
	X(t)= P(t) e^{Ft}.
	\end{equation}
	In particular, $P(0)=\mathrm{Id}_n$ and $X(T)=e^{FT}$.
\end{thm}

\begin{proof}
Let $X(t)$ be the fundamental matrix satisfying (\ref{ode}). Then $X(t+T)$ is also a fundamental matrix since $\frac{d}{dt}X(t+T)=A(t+T)X(t+T)=A(t)X(t+T)$. Hence, there exists an invertible matrix $C=X(T)$ such that $X(t+T)=X(t)C$. By the existence of matrix logarithm we can write $C=e^{FT}$. Let $$P(t):=X(t)e^{-Ft}.$$ Then $P(t+T)=X(t)Ce^{-FT}e^{-Ft}=X(t)e^{-Ft}=P(t)$.
\end{proof}

\begin{definition}
A function $x(t)$ is $\omega$-quasiperiodic with period $T$ if $x(t+T)=e^{iwT}x(t)$.
\end{definition}

\begin{prop}
With the same setting as in (\ref{floquetthm}), for each eigenvalue $f:=i\omega$ of $F$, there exists an $\omega$-quasiperiodic solution $x(t)$ to (\ref{ode}). 
\end{prop}
\begin{proof}
Assume that $i\omega$ is an eigenvalue of $F$ with eigenvector $v$. Then $\rho:=e^{i\omega}$ is an eigenvalue of $e^F$ with the same eigenvector. Let $x(t)=X(t)v=P(t)e^{Ft}v=P(t)\rho^tv=e^{i\omega t}P(t)v$. Then $x(t)$ is $\omega$-quasiperiodic with period $T$.
\end{proof}
\noindent
Observe that $\omega$ is defined modulo the frequency of modulation $\Omega$ given by $$\Omega:= 2\pi/T.$$ Therefore, we define the time-Brillouin zone as $Y_t^*:=\mathbb{C}/(\Omega\mathbb{Z})$.
For each eigenvalue $f:=i\omega$ of F, there is a Bloch solution $x(t)$ which is $\omega$-quasiperiodic.

\begin{definition}
	$e^{i\omega}$ is called  a characteristic multiplier. We refer to $\omega$ as a quasifrequency and to $f:=\iu\omega$ as a \textit{Floquet exponent}.
\end{definition}
\noindent
Note that if $A$ is time independent, then the solution to \eqref{ode} can be written as $x(t)=e^{At}x(0)$. The Floquet exponents are then given in this case by the eigenvalues of $A$. 
\noindent
The fact that the Floquet exponents are defined up to modulo $\iu\Omega$ leads us to the following definition.

\begin{definition}[Folding number]
\label{foldingnumber}
	Let $\omega_F$ be the imaginary part of an eigenvalue of the time independent matrix $F$. Then, we can uniquely write $\omega_F=\omega_0+m\Omega$, where $\omega_0\in [-\Omega/2,\Omega/2)$. The integer $m$ is called the \textit{folding number}.	
\end{definition}

\begin{definition}[Floquet transform]
        Given the lattice $\Lambda$, the dual lattice $\Lambda^*$, $Y^*:=\R^d/\Lambda^*$, and a function $f(x)\in L^2(\R^d)$, the Floquet transform of $f$ is defined as
        $$\FF[f](x,\al):=\sum_{m\in\Lambda}f(x-m)e^{i\al\cdot m}.$$
        Here, $\alpha\in Y^*$ is called \textit{quasiperiodicity} or \textit{quasimomentum}.
\end{definition}
Note that $\FF[f]$ is $\al$-quasiperiodic in $x$ with period $\Lambda$ and periodic in $\al$ with period $\Lambda^*$:
\begin{align*}
    \FF[f](x+l,\al) & =\sum_{m\in \Lambda}f(x+l-m)e^{\iu\al\cdot m}=e^{\iu\al\cdot l}\FF[f](x,\al),\quad\forall\, l\in\Lambda,\\
    \FF[f](x,\al+\beta) & =\sum_{m\in\Lambda}f(x-m)e^{-\iu\al\cdot\beta}=\FF[f](x,\al),\quad\forall\, \beta\in \Lambda^*.
\end{align*}
Moreover, the Floquet transform $\FF:L^2(\R^d)\to L^2(Y\times Y^*)$ is an invertible map whose inverse is given by $$\FF^{-1}[g](x)=\frac{1}{|Y^*|}\int_{Y^*}g(x,\al)d\al,\quad x\in\R^d.$$
Applying Floquet transform to the variable $x$ and seeking solutions quasiperiodic in $t$, the wave equation (\ref{waveequation}) becomes
 \begin{equation} \label{eq:wave_transf}
 	\begin{cases}\ \ds \left(\frac{\p }{\p t} \frac{1}{\kappa(x,t)} \frac{\p}{\p t} - \nabla \cdot \frac{1}{\rho(x,t)} \nabla\right) u(x,t) = 0,\\[0.3em]
 		\	u(x,t)e^{-\iu \alpha\cdot x} \text{ is $\Lambda$-periodic in $x$,}\\
 		\	u(x,t)e^{-\iu \omega t} \text{ is $T$-periodic in $t$}. 
 	\end{cases}
 \end{equation} 
For a given $\alpha\in Y^*$, we seek $\omega\in Y_t^*$ such that there is a non-zero solution $u$ to (\ref{eq:wave_transf}).
\begin{definition}
	Quasifrequencies $\w$ can be viewed as functions $\alpha \mapsto \omega(\alpha)$ and are called \textit{band functions}. Together, the $n$ band functions constitute the \textit{band structure} or \textit{dispersion relationship} of the material.
\end{definition}
\noindent The \textit{quasiperiodicity} or \textit{quasimomentum} $\alpha$ corresponds to the direction of wave propagation, and we therefore introduce the following definition.
\begin{definition}
 	Waves propagate reciprocally if for every $\alpha \in Y^*$, the set of quasifrequencies of  (\ref{eq:wave_transf}) at $\alpha$ coincides with the set of quasifrequencies at $-\alpha$. The reciprocal equation associated to  (\ref{eq:wave_transf}) for $\alpha\in Y^*$ is that with $-\alpha$.
\end{definition}

\subsection{Layer potentials and the quasiperiodic capacitance matrix}

	Let $\alpha\notin Y^*\setminus \{0\}$ and let $G$ be the Green's function for the Laplace equation in $d=2$ or $d=3$. Then we can define the quasiperiodic Green's function $G^{\alpha}(x)$ as the Floquet transform of $G(x)$, i.e., 
	\begin{equation}\label{eq:xrep}
		G^{\alpha}(x) := \sum_{m \in \Lambda} G(x-m)e^{\iu \alpha \cdot m}.
	\end{equation}
	The series in (\ref{eq:xrep}) converges uniformly for $x$ and $y$ in compact sets of $\mathbb{R}^d$, $x\neq y$ and $\alpha \neq 0$.  The quasiperiodic single layer potential  $\mathcal{S}_D^{\alpha}$ is defined as
	$$\mathcal{S}_D^{\alpha}[\varphi](x) := \int_{\partial D} G^{\alpha} (x-y) \varphi(y) \dx\sigma(y),\quad x\in \mathbb{R}^d.$$
	
	\begin{definition}[Quasiperiodic capacitance matrix] \label{QCM}
		Assume $\alpha \neq 0$. For a system of $N\in\N$ resonators $D_1,\dots,D_N$ in $Y$, we define the quasiperiodic capacitance matrix $C^\alpha=(C^\alpha_{ij})\in\mathbb{C}^{N\times N}$ to be the square matrix given by
		\begin{equation} \label{defci}
			C^\alpha_{ij}=-\int_{\p D_i} (\mathcal{S}_D^{\alpha})^{-1}[\chi_{\p D_j}] \dx \sigma,\quad i,j=1,\dots,N,
		\end{equation}
		where $\chi_{\p D_j}$ is the characteristic function of $\partial D_j$.
	\end{definition}
	At $\alpha = 0$, we can define $C^0$ by the limit of $C^\alpha$ as $\alpha \to 0$;  see \cite{ammari2021functional} for the details.
	
	The following variational representation of the quasiperiodic capacitance matrix
	holds. 
\begin{lemma}
\label{thm:q_cap}
    \cite{ammari2020robust} 
    Assume without loss of generality that $d=3$ and the periodicity direction to be the $x_1$-direction. Equivalently to (\ref{defci}), the quasiperiodic capacitance matrix can be defined as
    \begin{equation}
        C_{i j}^{\alpha}:=\int_{Y \backslash D} \overline{\nabla V_{i}^{\alpha}} \cdot \nabla V_{j}^{\alpha} \mathrm{d} x,
    \end{equation}
    where
    \begin{equation}
        \begin{cases}\Delta V_{j}^{\alpha}=0 & \text { in } Y \backslash D, \\ V_{j}^{\alpha}=\delta_{i j} & \text { on } \partial D_{i}, \\ e^{-\mathrm{i} \alpha x_{1}} V_{j}^{\alpha}\left(x_{1}, x_{2}, x_{3}\right) & \text { is periodic in } x_{1}, \\ V_{j}^{\alpha}\left(x_{1}, x_{2}, x_{3}\right)=O\left(\frac{1}{\sqrt{x_{2}^{2}+x_{3}^{2}}}\right) & \text { as } \sqrt{x_{2}^{2}+x_{3}^{2}} \rightarrow \infty, \text { uniformly in } x_{1}.\end{cases}
    \end{equation}
\end{lemma}
Based on Lemma \ref{thm:q_cap}, we obtain the following result. 
\begin{prop}
\label{thm:pos}
    For $\alpha\neq0$, the quasiperiodic capacitance matrix is positive definite.
\end{prop}
\begin{proof}
    Let $\zeta\in\mathbb{C}^N$. Using the definition of $C^\alpha$ given by Lemma \ref{thm:q_cap}, we obtain that
    \begin{equation}
        \begin{split}
            \overline{\zeta}C^\alpha \zeta & = \sum_{i=1}^N\sum_{j=1}^N\overline{\zeta_i}\zeta_jC_{ij}^\alpha \\
            & = \int_{Y \backslash D} \sum_{i=1}^N\overline{\zeta_i\nabla V_{i}^{\alpha}} \cdot  \sum_{j=1}^N\zeta_j \nabla V_{j}^{\alpha} \mathrm{d} x \\
            & = \int_{Y \backslash D} \lvert \sum_{i=1}^N{\zeta_i\nabla V_{i}^{\alpha}}\rvert^2 \mathrm{d}x.
        \end{split}
    \end{equation}
    This term is non-negative. Suppose now it is equal to zero, then it follows that
    $V^\alpha:=  \sum_{j=1}^N\zeta_j V_{j}^{\alpha}$ is constant in $Y \backslash D$. But  $V^\alpha$ is 
    $\alpha$-quasiperiodic and therefore, it should be zero. 
\end{proof}

\subsection{Time modulated subwavelength resonators}
For the purpose of this paper, we apply time modulations to the interior of the resonators, while the parameters of the surrounding material are constant in $t$. We let
\begin{equation}
	\label{modulation}
	\kappa(x,t)=
\begin{cases}
	\kappa_0, \ & x\in\mathbb{R}^d\backslash \overline{\mathcal{C}}\\
	\kappa_r\kappa_i(t),\ & x\in\mathcal{C}_i
\end{cases} 
,\ \ 
	\rho(x,t)=
\begin{cases}
	\rho_0, \ & x\in\mathbb{R}^d\backslash \overline{\mathcal{C}}\\
	\rho_r\rho_i(t),\ & x\in\mathcal{C}_i
\end{cases},
\end{equation}
for $i=1,\ldots, N$. Here, $\rho_0$, $\kappa_0$, $\rho_r$, and $\kappa_r$ are positive constants. The functions $\rho_i(t) \in \mathcal{C}^0(\mathbb{R})$ and $\kappa_i(t) \in \mathcal{C}^1(\mathbb{R})$ describe the modulation inside the $i^{\text{th}}$ resonator $\mathcal{C}_i$. Furthermore, we assume that $\rho_i,\kappa_i$ are periodic with period $T$.
\noindent
We define the contrast parameter $\delta$ as 
$$
\delta := \frac{\rho_r}{\rho_0}.
$$
In (\ref{waveequation}), the transmission conditions at $x\in \p D_i$ are
$$ u \big|_+ = u\big|_- \quad \mbox{and} \quad \delta \frac{\partial {u}}{\partial \nu} \bigg|_{+} - \frac{1}{\rho_i(t)}\frac{\partial {u}}{\partial \nu} \bigg|_{-} = 0, \qquad x\in \p D_i, \ t\in \R,$$
where $\partial/\partial \nu$ is the outward normal derivative at $\p D_i$ and $|_{+,-}$ denote the limits from outside and inside $D_i$, respectively. 

In order to achieve subwavelength resonance, we assume that $\delta \ll 1$ and consider regimes where the modulation frequency \begin{equation}
\label{subr} \Omega \big(:= \frac{2\pi}{T}\big) = O(\delta^{1/2}).\end{equation}
We also assume that $\kappa_i'(t)= O(\delta^{1/2})$ for $i=1,\ldots,N.$

Note that in the static case where $\rho_i(t)=\kappa_i(t)=1$ for all $i$, the system of $N$ subwavelength resonators has $N$ {\em subwavelength resonant frequencies} of order  $O(\delta^{1/2})$. We refer the reader to \cite{ammari2021functional} for more details. Note also that (\ref{subr}) allows strong coupling between the time modulations and the response time of the structure \cite{reviewTM2}.

	\label{setting}
We seek solutions to  (\ref{eq:wave_transf}) with modulations given by  (\ref{modulation}). Since $e^{-\mathrm{i}\omega t}u(x,t)$ is a $T$-periodic function of $t$, we can write, using Fourier series expansion,
$$u(x,t)= e^{\iu \omega t}\sum_{n = -\infty}^\infty v_n(x)e^{\iu n\Omega t}.$$
We then have from (\ref{eq:wave_transf}) the following equation in the frequency domain for $n\in \Z$:
\begin{equation} \label{eq:freq}
	\left\{
	\begin{array} {ll}
		\ds \Delta {v_n}+ \frac{\rho_0(\omega+n\Omega)^2}{\kappa_0} {v_n}  = 0 & \text{in } Y \setminus \overline{D}, \\[0.3em]
		\ds \Delta v_{i,n}^* +\frac{\rho_r(\omega+n\Omega)^2}{\kappa_r} v_{i,n}^{**}  = 0 & \text{in } D_i, \\
		\nm
		\ds  {v_n}|_{+} -{v_n}|_{-}  = 0  & \text{on } \partial D, \\
		\nm
		\ds  \delta \frac{\partial {v_n}}{\partial \nu} \bigg|_{+} - \frac{\partial v_{i,n}^* }{\partial \nu} \bigg|_{-} = 0 & \text{on } \partial D_i, \\[0.3em]
		v_n(x)e^{\iu \alpha\cdot x} \text{ is $\Lambda$-periodic in $x$}.
	\end{array}
	\right.
\end{equation}
Here, $v_{i,n}^*(x)$ and $v_{i,n}^{**}(x)$ are defined through the convolutions
$$v_{i,n}^*(x) = \sum_{m = -\infty}^\infty r_{i,m} v_{n-m}(x), \quad  v_{i,n}^{**}(x) = \frac{1}{\omega+n\Omega}\sum_{m = -\infty}^\infty k_{i,m}\big(\omega+(n-m)\Omega\big)v_{n-m}(x),$$
where $r_{i,m}$ and $k_{i,m}$ are the Fourier series coefficients of $1/\rho_i$ and $1/\kappa_i$, respectively:
$$\frac{1}{\rho_i(t)} = \sum_{n = -\infty}^\infty r_{i,n} e^{\iu n \Omega t}, \quad \frac{1}{\kappa_i(t)} = \sum_{n = -\infty}^\infty k_{i,n} e^{\iu n \Omega t}.$$
We can assume that the solution is normalized, i.e., $\|v_0\|_{H^1(Y)} = 1$. 
Here, $H^1$ is the usual Sobolev space of square-integrable functions whose weak derivative is square integrable.  Since $u$ is continuously differentiable in $t$, we then have as $n\to \infty$,
\begin{equation} \label{eq:reg_v}
	\|v_n\|_{H^1(Y)} = o\left(\frac{1}{n}\right).
\end{equation}  
We will consider the case where the modulation of $\rho$ and $\kappa$ has a finite Fourier series with a large number of nonzero Fourier coefficients: 
$$\frac{1}{\rho_i(t)} = \sum_{n = -M}^M r_{i,n} e^{\iu n \Omega t}, \qquad \frac{1}{\kappa_i(t)} = \sum_{n = -M}^M k_{i,n} e^{\iu n \Omega t},$$
for some $M\in \N$ satisfying
$$M = O\left(\delta^{-\gamma/2}\right),$$
where $0 < \gamma < 1$. We seek subwavelength quasifrequencies $\omega$ of the wave equation \eqref{eq:wave_transf} in the sense of the following definition first introduced in \cite{ammari2020time}.

 \begin{definition}[Subwavelength quasifrequency] \label{def:sub}
 		A quasifrequency $\omega = \omega(\delta) \in Y^*_t$ of \eqref{eq:wave_transf} is said to be a \emph{subwavelength quasifrequency} if there is a corresponding Bloch solution $u(x,t)$, depending continuously on $\delta$, which is essentially supported in the low-frequency regime, i.e., it can be written as
 		$$u(x,t)= e^{\iu \omega t}\sum_{n = -\infty}^\infty v_n(x)e^{\iu n\Omega t},$$
 		where 
 		$$\omega \rightarrow 0 \ \text{and} \ M\Omega \rightarrow 0 \ \text{as} \ \delta \to 0,$$
 		for some integer-valued function $M=M(\delta)$ such that, as $\delta \to 0$, we have
 		$$\sum_{n = -\infty}^\infty \|v_n\|_{L^2(Y)} = \sum_{n = -M}^M \|v_n\|_{L^2(Y)} + o(1).$$
 	\end{definition}
\noindent In particular, one can prove that if  $\omega$ is a subwavelength quasifrequency, then $\omega$ and the frequency of modulation $\Omega$ have the same order:
$$\omega = O\left(\delta^{1/2}\right).$$ 	
%
\noindent
Next, we introduce the capacitance matrix characterization of the band structure of time dependent periodic systems of subwavelength resonators. As shown below, $C^\alpha$ offers a discrete approximation to subwavelength quasifrequency problems. 
More details can be found in \cite{ammari2020time}. 
\begin{thm} \label{thm:pre}
		As $\delta \to 0$, the subwavelength quasifrequencies of the wave equation (\ref{eq:wave_transf}) are, up to leading-order, given by the quasifrequencies of the system of ODEs:
		\begin{equation}
		\label{Hill}
			\frac{\dx^2\phi}{\dx t^2}(t)+M^\alpha(t)\phi(t)=0,	
		\end{equation}
		where $M^\alpha$ is the matrix defined as
		\begin{equation}
			M^\alpha(t)=\frac{\delta\kappa_r}{\rho_r}W_1(t)C^\alpha W_2(t)+W_3(t),
		\end{equation}
		and $W_1,W_2$ and $W_3$ are the diagonal given by
		\begin{equation}
			(W_1)_{ii}=\frac{\sqrt{\kappa_i}\rho_i}{\lvert D_i\rvert},\quad (W_2)_{ii}=\frac{\sqrt{\kappa_i}}{\rho_i},\quad (W_3)_{ii}=\frac{\sqrt{\kappa_i}}{2}\frac{\dx}{\dx t}\frac{\dx \kappa_i/dt}{\kappa_i^{3/2}}.
		\end{equation}
		Here, $\lvert D_i\rvert$ denotes the volume of $D_i$. 
\end{thm}

\begin{lemma} Let $\LL^\al$ be the operator associated with the system of Hill's equations 
$$\LL^\alpha[\phi]:= \frac{\dx^2\phi}{\dx t^2}(t)+M^\alpha(t)\phi(t)=0.$$
If $\omega$ is a quasifrequency associated to $\LL^\alpha$, then $-\bar\w$ is a quasifrequency associated to $\LL^{-\al}$.
\end{lemma}
\begin{proof}
\label{bandsym}
Since the quasiperiodic capacitance matrix $C^\al$ is Hermitian (see, for instance, \cite{ammari2021functional})  and the entries of $W_i$ $(i=1,2,3)$ are all real, $M^{-\al}=\overline{M^{\al}}$. If $\w$ is a quasifrequency associated with $\al$, i.e., if there exists a non-trivial solution $\phi$ to (\ref{Hill}) with $\phi(t+T)=e^{i\w T}\phi(t)$, then $-\w$ is a quasifrequency associated with $-\al$, where the corresponding solution is given by $\overline{\phi}$ and $\overline\phi(t+T)=e^{-i\w T}\overline\phi(t)$.
\end{proof}

\section{Asymptotic analysis of the quasifrequencies and its implications}
\label{sec2}
\subsection{Problem formulation}
In the following, we shall study the effect of small periodic perturbations of the material parameters $\rho(t)$ and $\kappa(t)$ on the subwavelength quasifrequencies. Similar to perturbation theory in quantum mechanics \cite{perturbation}, we assume that $M^\alpha(t)$ is an analytic function of $\varepsilon$ at $\varepsilon=0$ and can be written as 
\begin{equation} \label{expm}
M^\alpha(t) = M^\alpha_0 + \varepsilon M_1^\alpha(t) + \ldots + \varepsilon^n M_n^\alpha(t) + \ldots,
\end{equation}
where $\varepsilon > 0$ is some small parameter describing the amplitude of the time modulation and $ M^\alpha_0$ corresponds to the unmodulated case. We can assume that the above series converges for $|\varepsilon| < \eps_0$, where $\eps_0 >0$ is independent of $t$, provided that the modulations of $\rho$ and $\kappa$ have finitely many non-zero Fourier coefficients \cite{jinghaothesis,TheaThesis}. 

We shall make use of the asymptotic Floquet analysis developed in \cite{TheaThesis}, which is a combination of perturbation analysis and Floquet theory; see also \cite{Yakubovich}. We can rewrite the second-order ODE \eqref{Hill} into
\begin{equation}\label{eq:1dsys}
	\frac{\dx y}{\dx t}(t) = {A}(t)y(t), \qquad {A}(t):=\begin{pmatrix}
	0 & \mathrm{Id}_N\\
	-M^\alpha(t) & 0
\end{pmatrix}.
\end{equation}
We will focus on the asymptotic analysis of the quasifrequencies associated with $\pm\alpha$ in terms of $\varepsilon$.
When the setting is clear, we shall omit the superscript $\al$ and write the fundamental solution as $X(t)=P(t)e^{Ft}$ (c.f. (\ref{Floquettheory})). The reader shall bear in mind that $A(t)$, $P(t)$ and $F$ all depend on $\al$.
We can then write \cite{jinghaothesis,TheaThesis}
\begin{equation} 
\label{aexp}
\left\{
\begin{array}{lll}
    A(t)&=&A_0+\eps A_1(t) +\cdots +\eps^nA_n(t)+\cdots, \\
    \nm
    P(t)&=& P_0+\eps P_1(t) +\cdots +\eps^nP_n(t)+\cdots, \\
    \nm
    F&=& F_0+\eps F_1+\cdots + \eps^nF_n+\cdots . 
\end{array}
\right.
\end{equation}
Since $A(t)$ is $T$-periodic (c.f.(\ref{ode})), the $A_n$'s are also $T$-periodic for all $n\geq 1$. With $\Omega={2\pi}/{T}$, we can write 
\begin{equation}
    A_n=\sum_{m\in\Z}A_n^{(m)}e^{\iu \Omega m}.
\end{equation}
Asymptotic analysis amounts to explicitly expanding the Floquet matrix $F$ in terms of the amplitude of modulation $\eps$, and then applying eigenvalue perturbation theory to derive asymptotic formulas of the form  $f=f_0+\eps f_1+\cdots$ for 
the eigenvalues  $f$ of $F$. In the following sections, we will explicitly compute the quasifrequencies up to first-order in $\eps$.
Assume now that $A_0$ defined in (\ref{aexp}) is diagonal. Let $m_i$ be the folding number of $(A_0)_{ii}$. Then 
\begin{equation} \label{F0}
F_0=A_0-i\Omega\begin{psmallmatrix}
m_1 &&\\&\ddots&\\&&m_n\end{psmallmatrix}.
\end{equation}

The most interesting case for us is to consider the perturbations due to the modulations at a degenerate point $f_0$ of $F_0$ (with corresponding eigenspace of dimension higher than one). In view of (\ref{F0}), such degeneracy can be obtained through folding \cite{ammari2020time}. 

We need the following lemma. 
\begin{lemma} \label{lemma31}  With the same notation as above, we have
\begin{itemize}
\item For every $j$, $(F_1)_{jj}=(A_1^{(0)})_{jj}$;
\item If $(F_0)_{ll}=(F_0)_{jj}$ for some $l\neq j$, then $(F_1)_{jl}=(A_1^{(m_l-m_j)})_{jl}$. More generally, for $l\neq j$, we have
  \begin{equation}
      (F_1)_{jl}=
    \begin{cases}
        \left(A_1^{(m_j-m_l)}\right)_{jl} & \text{if } (F_0)_{jj}=(F_0)_{ll},\\
        \left((F_0)_{ll}-(F_0)_{jj}\right)\sum_{m\in\Z}\frac{\left(A_1^m\right)_{jl}}{i\Omega m+(A_0)_{ll}-(A_0)_{jj}} &\text{otherwise}.
    \end{cases}
    \end{equation}
\end{itemize}
\end{lemma}\noindent
Proofs of the results stated in Lemma \ref{lemma31} can be found in \cite{jinghaothesis,TheaThesis}.
The following theorem can be directly derived from the above lemma using the eigenvalue perturbation theory described in the appendix. 
\begin{thm}
\label{thm:floquet_matrix_elements}
Let $f_0$ be a degenerate point with multiplicity $r$. Then $F$ has associated eigenvalues given by 
\begin{equation}
\label{eigenexp}
f_0 + \varepsilon f_i + O(\varepsilon^2),
\end{equation}
where $f_i$, for $i=1,\ldots,r$, are the eigenvalues of the $r\times r$ upper-left block of $F_1$, whose entries are given by
	\begin{equation} 
	\label{eq:matrixA1}
		(F_1)_{lk}=\left(A_1^{(m_l-m_k)}\right)_{lk} \ \text{for} \ l,k = 1,\ldots,r,
	\end{equation}
	with  $m_l$ and $m_k$ denoting the folding numbers of the $l$-th and $k$-th eigenvalues of $A_0$.
\end{thm}

\begin{cor}
\label{example_f1}
If the degenerate points are of order $r=2$ and $A_1$ has no constant part; i.e., $A_1^{(0)} = 0$, then 
the leading-order terms $(f_i)_{i=1,2}$ in (\ref{eigenexp}) are given by the eigenvalues of the matrix
\begin{equation}
	\label{pvalue}
	\begin{pmatrix}
		0 & (F_1)_{12}\\
		(F_1)_{21} & 0	
	\end{pmatrix}.
\end{equation}
In particular, the eigenvalues $f$ of $F$ associated with the degenerate point $f_0$ are given by 
\begin{equation}\label{eq:pert}
	f = f_0 \pm \varepsilon\sqrt{(F_1)_{12}(F_1)_{21}} + O(\varepsilon^2).
\end{equation}
\end{cor}
Note that in the non-degenerate case, the leading-order term in $\varepsilon$ vanishes and the perturbations of the Floquet exponents that are due to time modulations with amplitude $\varepsilon$ are quadratic in $\varepsilon$ \cite{TheaThesis}. 

\subsection{Floquet matrix elements}
\label{sec3}
From now on, we only focus on degenerate points with multiplicity $2$ and consider the case where $\kappa_i(t)$ and $\rho_i(t)$ are of the form
\begin{equation} 
\frac{1}{1+\varepsilon \text{cos}(\Omega t + \phi_i)}, \quad i=1, \ldots, N.
\end{equation}
Here, $\Omega$ is the frequency of the modulation, $\varepsilon$ is the modulation amplitude, and $\phi_i$ is a phase shift. From (\ref{eq:matrixA1}) and (\ref{eq:pert}), non-zero Fourier coefficients of $A_1$ are needed to compute the first-order perturbation of the quasifrequencies. This boils down to the computation of the non-zero Fourier coefficients of $M^{\alpha}_1$, where $M^{\alpha}_1$ is the first-order expansion of the matrix $M^\alpha$ with respect to $\varepsilon$, as seen in (\ref{eq:1dsys}). From now on, we fix $\alpha\in Y^*$ and omit the superscript $\alpha$.
\begin{thm}
In the case where the $\rho_i$'s are constant and the $\kk_i$'s are time modulated: $$\kk_i(t):=\frac{1}{1+\eps \text{cos}(\Omega t+\vp_i)}.$$ The only non-zero Fourier coefficients of $M_1$ defined in (\ref{expm}) are given by
\begin{equation}
    \begin{split}
   \left( M_1^{(1)}\right)_{ii}=\left(\frac{\Omega^2}{2}-L_{ii}\right)\frac{e^{i\vp_i}}{2}, 
   & \left( M_1^{(1)}\right)_{ij}=-L_{ij}\cdot \frac{e^{i\vp_i}+e^{i\vp_j}}{4},\\
   \left( M_1^{(-1)}\right)_{ii}=\left(\frac{\Omega^2}{2}-L_{ii}\right)\frac{e^{-i\vp_i}}{2},
   & \left( M_1^{(-1)}\right)_{ij}=-L_{ij}\cdot \frac{e^{-i\vp_i}+e^{-i\vp_j}}{4}.
\end{split}
\end{equation}
\end{thm}
\begin{proof}
In the case where only the $\kappa_i$'s are time modulated, $M$ is given by $M=K\widetilde C^\al+W_3$, where $K=\delta {\rho_r}/{\kk_r}$ is a constant and
\begin{align*}
    \widetilde{C}^\al_{ij}  & =\frac{C^\al_{ij}}{|D_i|}\frac{1}{\sqrt{(1+\eps \text{cos}(\Omega t+\vp_i))(1+\eps \text{cos}(\Omega t+\vp_j))}}\\& =\frac{C^\al_{ij}}{|D_i|}\left(1-\frac{1}{2}\eps\left(\text{cos}(\Omega t+\vp_i)+\text{cos}(\Omega t+\vp_j)\right)+O(\eps^2)\right),\\
     (W_3)_{ii} & =\frac{\Omega^2}{4}\left(1+\frac{\eps^2-1}{(1+\eps \text{cos}(\Omega t+\vp_i)^2}\right)=\frac{\Omega^2}{2}\eps \text{cos}(\Omega t+\vp_i)+O(\eps^2).
\end{align*}
Defining $L_{ij}:={KC^\al_{ij}}/{|D_i|}$, we get
\begin{equation}
    M_{ij}=
    \begin{cases}
        L_{ij}-\frac{\eps}{2}L_{ij}\left(\text{cos}(\Omega t+\vp_i)+\text{cos}(\Omega t+\vp_j)\right)+O(\eps^2),&i\neq j,\\
        \nm
        L_{ii}+\eps\left(\frac{\Omega^2}{2}-L_{ii}\right)\text{cos}(\Omega t+\vp_i) +O(\eps^2),&i=j.
    \end{cases}
\end{equation}

\noindent Writing $M(t)=M_0+M_1(t)\eps+O(\eps^2)$, we see that $(M_0)_{ij}=K {C_{ij}^\alpha}/{|D_i|}$ depends only on $\alpha$ and corresponds to the unmodulated case $\eps = 0$. 
\end{proof}
\begin{thm}
\label{thm:modulatingboth}
In the case where the $\rho_i$'s and the $\kk_i$'s are all time modulated and defined as follows: 
$$\kk_i(t):=\frac{1}{1+\eps \text{cos}(\Omega t+\vp_i)},\quad \rho_i(t):=\frac{1}{1+\eps \text{cos}(\Omega t+\phi_i)}, \quad i=1,\ldots, N, $$
the non-zero Fourier coefficients of $M_1$ are given by
\begin{align}
  \left( M_1^{(1)}\right)_{ii}&=\left(\frac{\Omega^2}{2}-L_{ii}\right)\frac{e^{i\vp_i}}{2}, 
  & \left( M_1^{(1)}\right)_{ij}&=L_{ij}\left(\frac{e^{i\phi_i}-e^{i\phi_j}}{2}+ \frac{e^{i\vp_i}+e^{i\vp_j}}{4}\right),\\
  \left( M_1^{(-1)}\right)_{ii}&=\left(\frac{\Omega^2}{2}-L_{ii}\right)\frac{e^{-i\vp_i}}{2},
  & \left( M_1^{(-1)}\right)_{ij}&=L_{ij}\left(\frac{e^{-i\phi_i}-e^{-i\phi_j}}{2}+ \frac{e^{-i\vp_i}+e^{-i\vp_j}}{4}\right),
\end{align}
where $L_{ij}:={KC^\al_{ij}}/{|D_i|}$.
\end{thm}
\begin{proof}
In this case, we have $M=K\widetilde C^\al+W_3$, where $K$ and $W_3$ are the same as before. Then, the following equation holds:
\begin{equation}
    \widetilde C^\al_{ij}=
    \begin{cases}
        \frac{C^\al_{ij}}{|D_i|}\sqrt{\kk_i(t)\kk_j(t)}\frac{\rho_i(t)}{\rho_j(t)},& i\neq j,\\
        \nm
        \frac{C^\al_{ii}}{|D_i|}\kk_i, & i=j.
    \end{cases}
\end{equation}
Hence, we obtain that
\begin{equation}
M_{i j}^{\alpha}= \begin{cases}L_{i j}+\varepsilon L_{i j}\left(\cos \left(\Omega t+\phi_{i}\right)-\cos \left(\Omega t+\phi_{j}\right)\right. & \\ -\frac{1}{2}\left(\cos \left(\Omega t+\varphi_{i}\right)+\cos \left(\Omega t+\varphi_{j}\right)\right)+O\left(\varepsilon^{2}\right), & i \neq j, \\ L_{i i}+\varepsilon\left(\frac{\Omega^{2}}{2}-L_{i i}\right) \cos \left(\Omega t+\varphi_{i}\right)+O\left(\varepsilon^{2}\right), & i=j .\end{cases}
\end{equation}
\end{proof}
\subsection{Numerical validation of the asymptotic formulas for the quasifrequencies emerging from a degenerate point}

The analytical results above can be verified using numerical simulations. We first choose the structure to be an infinite chain of trimers. In concrete, the resonators are modeled as disks of radius $0.1$, placed in a one dimensional lattice with period $L$ as shown in Figure \ref{figtrimer}.

\begin{figure}[H]
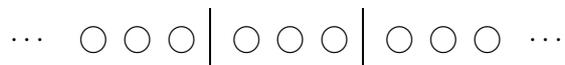

$$
\begin{array}{cc|c|cc}
\cdots & \;\bigcirc \; \bigcirc \; \bigcirc\; & \;\bigcirc \; \bigcirc \; \bigcirc\; &\; \bigcirc \; \bigcirc \; \bigcirc \;& \cdots\\[0.2em]
\end{array}
$$
\caption{A chain of trimers.} \label{figtrimer}
\end{figure}
The band structure of the chain of trimers depends on the modulation frequency $\Omega$. It exhibits degenerate points in the static case due to the folding in the time Brillouin zone. This is depicted in Figure \ref{fig:static_omega}. We then fix a positive degenerate point $\alpha=\alpha_\text{deg}$ and turn on the time modulations to observe the emergence of the quasifrequencies $\omega_1$ and $\omega_2$ from the degenerate point. Their real and the imaginary parts are plotted against the modulation amplitudes $\varepsilon_\rho$ and $\varepsilon_\kappa$. We then use polynomial interpolation on the numerical data to obtain the first-order perturbation coefficients for the case where we time modulate $\rho$ only, $\kappa$ only and both of them. This is then compared with the analytical results proven in Theorem \ref{thm:modulatingboth} in Table \ref{tbl:compare}.
\begin{table}[H]
\begin{center}
\begin{tabular}{|c!{\vrule width 1.5pt}C|C|C!{\vrule width 1.5pt}C|C|C|}
\hline
    \multicolumn{1}{|c!{\vrule width 1.5pt}}{$\Omega=0.3,\;\alpha_{deg}=2.395$}& \multicolumn{3}{|c!{\vrule width 1.5pt}}{$f_1$: analytical result} & \multicolumn{3}{|c|}{$f_1$: numerical result}\\
    \hline
    Quasifrequencies & \rho & \kappa & \text{both} & \rho & \kappa & \text{both} \\
    \wline{1.5pt}
    \makecell{$\omega_{1,2}=0.12+0\mathrm{i}$} & \pm 0.0096 & \pm 0.0116\mathrm{i}&\pm0.0064\mathrm{i}& \pm 0.0096 & \pm 0.0116\mathrm{i}&\pm0.0064\mathrm{i}\\
    \wline{1.5pt}

\end{tabular}
\end{center}
\caption{The numerical results agree with the results predicted by the proven asymptotic formulas.}
\label{tbl:compare}
\end{table}
\begin{figure}[H]
\centering
\begin{subfigure}[b]{.45\textwidth}
  \centering
  \includegraphics[width=\linewidth]{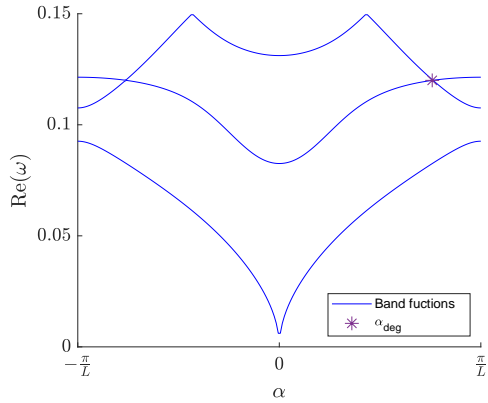}
  \caption{Band functions of the static trimers chain with $\Omega=0.3$, exhibiting degenerate points.}
  \label{fig:static_omega}
 \end{subfigure}
 \hspace{0.5cm}
\begin{subfigure}[b]{.45\textwidth}
  \centering
  \includegraphics[width=\linewidth]{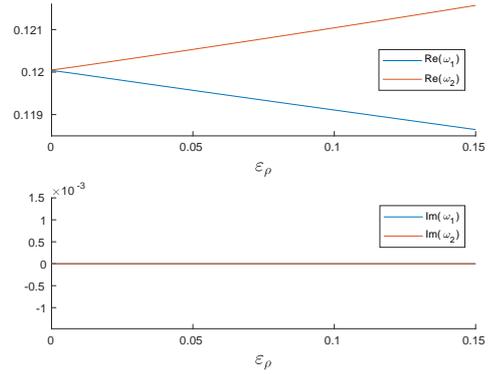}
  \caption{Quasifrequencies at $\alpha_\text{deg}$ when only modulating $\rho$. Here, $\varepsilon_\rho$ denotes the modulation amplitude.} 
 \end{subfigure}
\begin{subfigure}[b]{.45\textwidth}
  \centering
  \includegraphics[width=\linewidth]{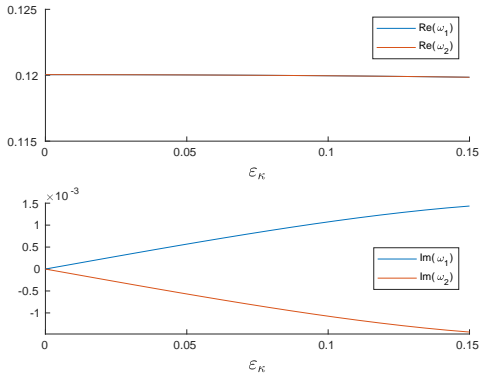}
  \caption{Quasifrequencies at $\alpha_\text{deg}$ when only modulating $\kappa$. Here, $\varepsilon_\kappa$ denotes the modulation amplitude.} 
 \end{subfigure}
 \hspace{0.5cm}
 \begin{subfigure}[b]{.45\textwidth}
  \centering
  \includegraphics[width=\linewidth]{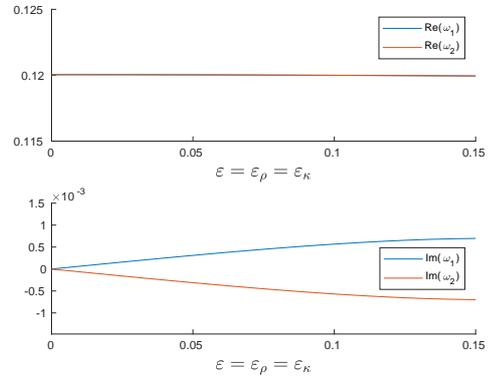}
  \caption{Quasifrequencies at $\alpha_\text{deg}$ when modulating both $\rho$ and $\kappa$.} 
 \end{subfigure}
\caption{Asymptotic analysis of the quasifrequencies of the trimer chain under 
time modulation at the degenerate point. The modulation frequency is $\Omega=0.3$.} 
\label{fig:modulation_formula}
\end{figure}

\subsection{Transmission properties under time modulations}
Our main result in this paper is the following theorem. 
\begin{thm}
\label{thm:purelyima}
Assume all the resonators are of the same size. If only one of the the material parameters ($\kappa$ or $\rho$) is modulated and the other one remains constant in time, then the first-order perturbation in the quasifrequencies at a degenerate point is either real or purely imaginary. 
\end{thm}
\begin{proof}



In the case where all the resonators have the same size, we have $M_0:=M_0^\alpha=kC^\al$ where $k>0$ is a constant and $C^\alpha$ is the capacitance matrix. Since $C^\al$ is Hermitian and positive definite (see Proposition \ref{thm:pos}), all the eigenvalues of $M_0$ are strictly positive and $M_0$ is unitarily diagonalizable. Let $$\ds \widetilde{A_0}:=\begin{pmatrix}0 &I_N\\-M_0&0\end{pmatrix}$$ be the $2N\times 2N$ square matrix and let $A_0=\widetilde S^{-1}\widetilde A_0 \widetilde S$ be its diagonal form. We observe that $\lambda$ is an eigenvalue of $\widetilde{A_0}$ if and only if $\mu:=-\lambda^2$ is an eigenvalue of $M_0$. Let $\lambda_1,\dots,\lambda_N$ be the eigenvalues of $\widetilde A_0$ of the form $\sqrt\mu \,\iu\in\iu\R_{>0}$. Denote by $D:=${\rm diag}$(\lambda_1,\ldots,\lambda_N)$ and $S$ the unitary matrix whose columns are eigenvectors of $M_0$. Then $S^{-1}M_0S=-D^2$ and one can check that with $\ds \widetilde S:=\begin{pmatrix} S&-S\\SD&SD\\\end{pmatrix}$, $A_0$ is diagonal:
$$\ds A_0:=\widetilde S^{-1} \widetilde A_0\widetilde S=\begin{pmatrix} D&0\\0&-D\end{pmatrix}.$$
Moreover, we have that
$$\ds \widetilde S^{-1}=\frac{1}{2}
\begin{pmatrix}
S^{-1}&D^{-1}S^{-1}\\-S^{-1}&D^{-1}S^{-1}
\end{pmatrix}.
$$
Since $A_1^{(m)}=0$ if $m\neq\pm 1$,  it suffices to consider $A_1^{(\pm 1)}$. We know that
$$A_1^{(\pm 1)}:=\widetilde S^{-1}\widetilde A_1^{(\pm 1)}\widetilde S=\frac{1}{2}\begin{pmatrix}
D^{-1}B^{(\pm 1)}&-D^{-1}B^{(\pm 1)}\\D^{-1}B^{(\pm 1)}&-D^{-1}B^{(\pm 1)}
\end{pmatrix},$$ where $B^{(\pm 1)}:=S^{-1}M_1^{(\pm 1)}S$. 
When only modulating $\rho$, we have that $M^{(-1)}=-\overline{M^{(1)}}^T$ and so  $B^{(-1)}=-\overline{B^{(1)}}^T$; while when modulating only $\kappa$, $M^{(-1)}=\overline{M^{(1)}}^T$ and  $B^{(-1)}=\overline{B^{(1)}}^T$. Hence, two cases arise: 
\par\noindent
\textbf{Case 1}: $(f_0)_{i}=(f_0)_{j}$ for some $1\leq i,j\leq N$ (or analogously $N\leq i,j\leq 2N$), i.e., $\exists \; i,j\leq N$ such that the corresponding eigenvalues of $A_0$ satisfy $\lambda_i\cong \lambda_j$ mod $\Omega$ . Then

$$(A_1^{(1)})_{ij}(A_1^{(-1)})_{ji}=\underbrace{\lambda_i^{-1}\lambda_j^{-1}}_{<0}B^{( 1)}_{ij}B^{(-1)}_{ji}=
\begin{cases}
    -\lambda_i^{-1}\lambda_j^{-1}|B^{(1)}_{ij}|^2>0 & \text{ when modulating } \rho ,\\
    \nm
    +\lambda_i^{-1}\lambda_j^{-1}|B^{(1)}_{ij}|^2<0 & \text{ when modulating } \kappa ;
\end{cases}$$

\noindent
\textbf{Case 2}: \label{case 2} $(f_0)_{i}=(f_0)_{j'}$ for some $1\leq i,j\leq N$ and $j'=N+j$, i.e., $\lambda_i=-\lambda_j$ mod $\Omega$.

$$(A_1^{(1)})_{ij'}(A_1^{(-1)})_{j'i}=-\underbrace{\lambda_i^{-1}\lambda_j^{-1}}_{<0}B^{(1)}_{ij}B^{(-1)}_{ji}=
\begin{cases}
    +\lambda_i^{-1}\lambda_j^{-1}|B^{(1)}_{ij}|^2<0 & \text{ when modulating } \rho ,\\
    \nm
    -\lambda_i^{-1}\lambda_j^{-1}|B^{(1)}_{ij}|^2>0 & \text{ when modulating } \kappa .
\end{cases}$$
We conclude the proof by inserting the different cases into the formulas for the first-order perturbation given by Theorem \ref{thm:floquet_matrix_elements} and Corollary \ref{example_f1}.
\end{proof}
\begin{rem}
When modulating both $\rho$ and $\kappa$, the expressions for $M_1^{(\pm 1)}$ and $B^{(\pm 1)}$ can be obtained from the expressions above. Denoting the formulas for modulating only one parameter by $B_\rho^{(\pm 1)}$ and $B_\kappa^{(\pm 1)}$, we can write
\begin{align*}
B_{ij}^{(1)}B_{ji}^{(-1)}
&=
\left(
(B_\rho^{(1)})_{ij}+(B_\kappa^{(1)})_{ij}\right)
\left(
(B_\rho^{(-1)})_{ji}+(B_\kappa^{(-1)})_{ji}\right)\\
&=
\left(
(B_\rho^{(1)})_{ij}+(B_\kappa^{(1)})_{ij}\right)
\left(
-(\overline{B_\rho^{(1)}})_{ij}+(\overline{B_\kappa^{(1)}})_{ij}\right)\\
&=
\left| (B_\kappa^{(1)})_{ij} \right|^2 -\left| (B_\rho^{(1)})_{ij}\right|^2+2\text{ Im} (z) \iu,
\end{align*}
where 
$z:=(\overline{B_\kappa^{(1)}})_{ij}(B_\rho^{(1)})_{ij}$. Numerical simulations show that the order of Im$(z)$ is considerably ($10^{-4}$ times) smaller than the term $| (B_\kappa^{(1)})_{ij} |^2 -| (B_\rho^{(1)})_{ij}|^2$.
\end{rem}
\begin{rem}
\label{rem:main}
In the high-contrast regime, we verify numerically that only the second case as discussed in the proof of Theorem \ref{thm:purelyima} occurs. Indeed, since $M_0=\delta ({\rho_r}/{\kappa_r}) \, \mathrm{Id}_N$ and $\delta\ll1$, the eigenvalues $\mu$ of $M_0$ are small enough so that $\text{ Im}(\lambda)=\sqrt{\mu}$ never exceed $\Omega$. Hence, when $A_0$ gets folded into $F_0$, an eigenvalue $\lambda$ with positive imaginary part either stays the same or gets folded into an eigenvalue with negative imaginary part and vice-versa. Thus, we can  conclude that modulating $\rho$ always yields imaginary $f_1$, leading to real perturbation in the quasifrequencies and modulating $\kappa$ always results in purely imaginary perturbations  in the quasifrequencies (at leading-order in the modulation amplitude $\varepsilon$).
\end{rem}
The following theorem is proved in \cite{jinghaothesis}.
\begin{thm}
\label{thm:bandgap}
If only the material parameters $\rho_i$'s are time modulated, then there is a non-reciprocal band gap opening around the degenerate point.
\end{thm}

We now prove the following result. 
\begin{thm}
\label{thm:kgap}
If only the material parameters $\kappa_i$ are time modulated, then at a degenerate point with multiplicity $2$, one of the two Bloch modes is exponentially decaying and the other is exponentially increasing over time. The momentum gaps where waves exhibit this exponential behavior are called the k-gaps.
\end{thm}
\begin{proof}
By arguments layed out in Remark \ref{rem:main}, the first-order perturbations of the quasifrequencies are given by $\pm\omega_1$, where $\omega_1\in \mathrm{i}\mathbb{R}$. This means that, the two subwavelength solutions in the case $u_\pm$ to \eqref{eq:wave_transf} satisfy
\begin{equation}
    u_\pm(x,t)e^{i(\omega_0\pm\varepsilon\omega_1)t} \ \text{is $T$-periodic},
\end{equation}
where $\omega_0\in\mathbb{R}$ and $\omega_1\in\mathrm{i}\mathbb{R}$. This implies that $u_+$ is exponentially increasing and $u_-$ is exponentially decreasing when $t\rightarrow + \infty$.
\end{proof}

\subsection{Numerical illustrations of non-reciprocal band gaps and k-gaps} 

The following figures illustrate numerically the results from Theorem \ref{thm:bandgap} and Theorem \ref{thm:kgap}. The occurrence of a bandgap signifies that waves with frequencies inside this band gap cannot propagate through the material and will decay exponentially. As shown in Figure \ref{fig:6}, there is a bandgap when modulating $\rho$ only, due to the fact that at the degenerate point both the constant term $\omega_0$ and the first-order term $\omega_1$ are real. In this case, the wave transmission is non-reciprocal as the size of the bandgap varies from left to right \cite{jinghaothesis}. However, there is no bandgap when modulating $\kappa$, since the first-order term is purely imaginary at the degenerate point. Nevertheless, there is a k-gap opening in this case and the structure can support exponentially growing waves with these quasiperiodicities. 

\begin{figure}[H]
\centering
\begin{subfigure}{.45\textwidth}
  \centering
  \includegraphics[width=\linewidth]{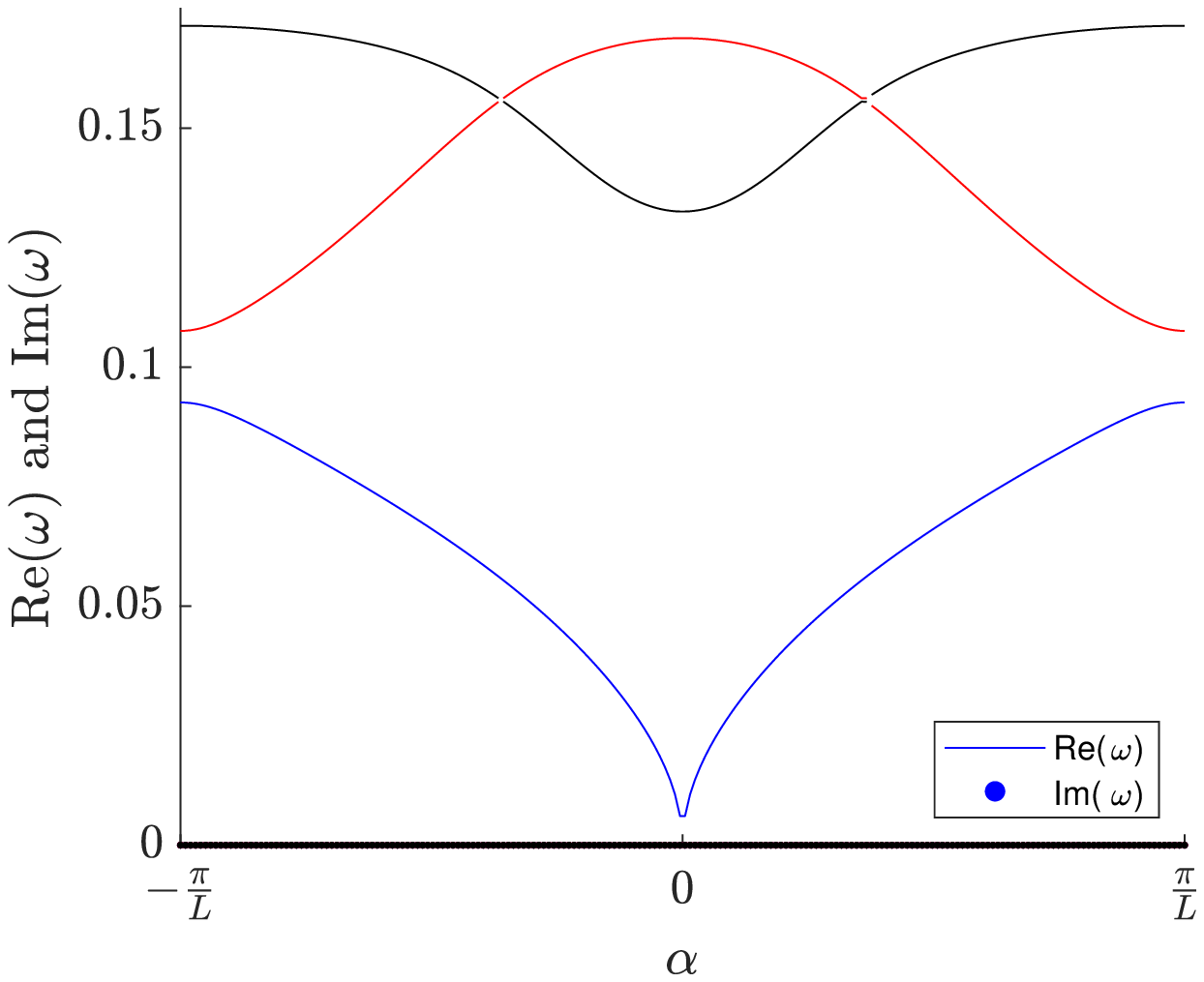}
  \caption{Degenerate points for the static case.}
 \end{subfigure}
 \hspace{0.5cm}
\begin{subfigure}{.45\textwidth}
  \centering
  \includegraphics[width=\linewidth]{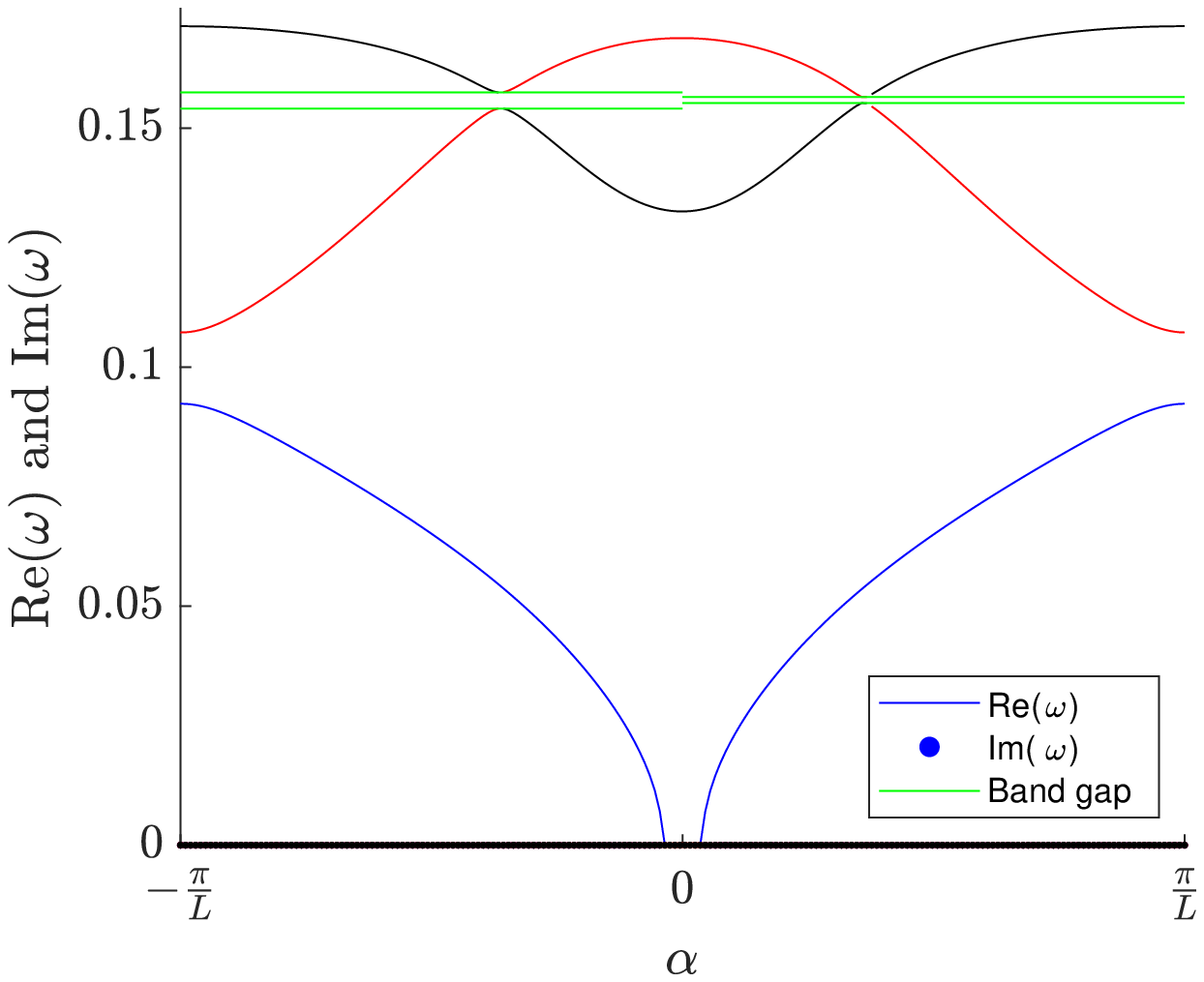}
  \caption{Appearance of band gaps by modulating $\rho$ only.} 
 \end{subfigure}
\\
\begin{subfigure}{.45\textwidth}
  \centering
  \includegraphics[width=\linewidth]{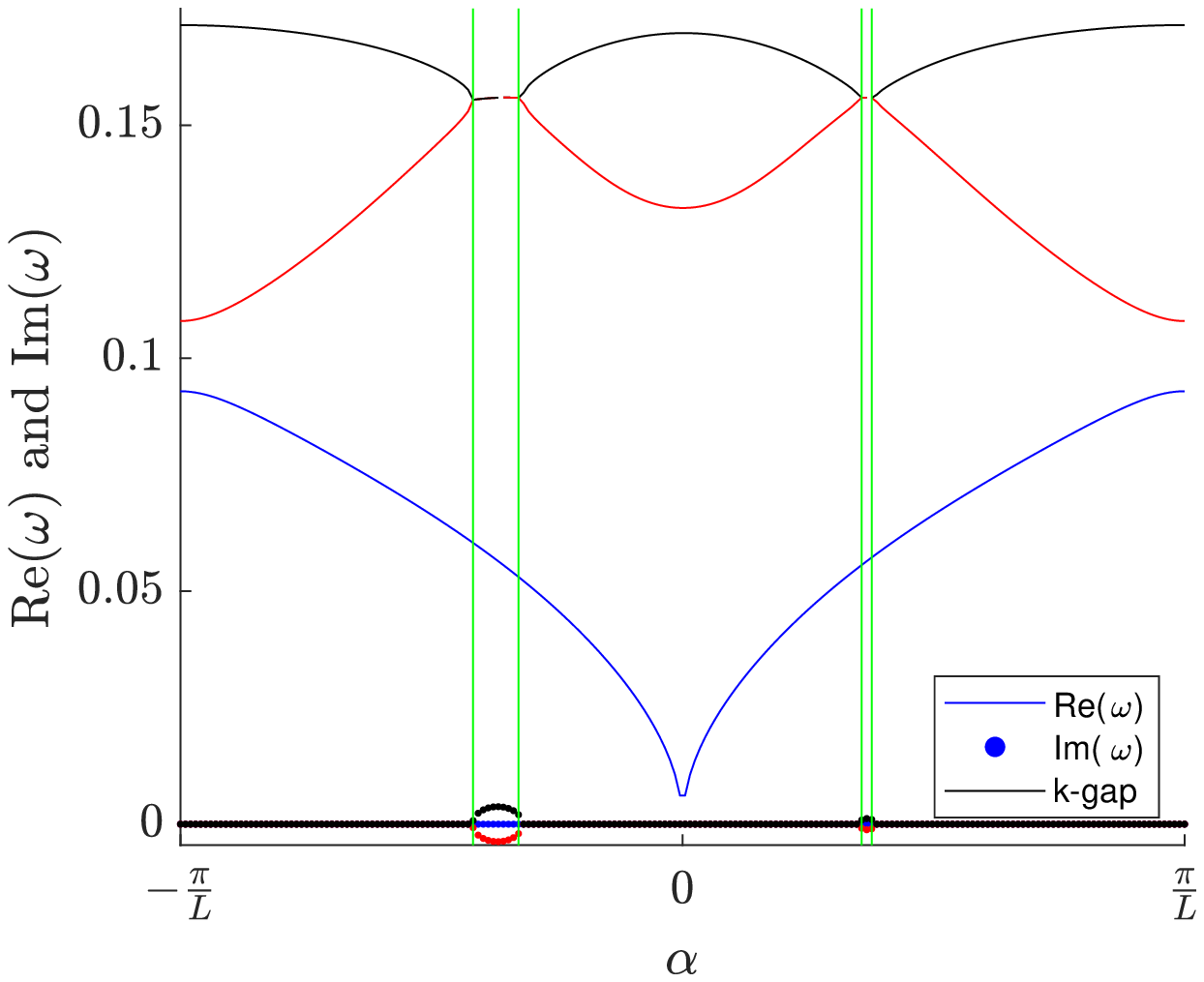}
  \caption{Appearance of k-gaps by modulating $\kappa$ only.} 
    \label{fig:kgap}
\end{subfigure}
\hspace{0.5cm}
\begin{subfigure}{.45\textwidth}
  \centering
  \includegraphics[width=\linewidth]{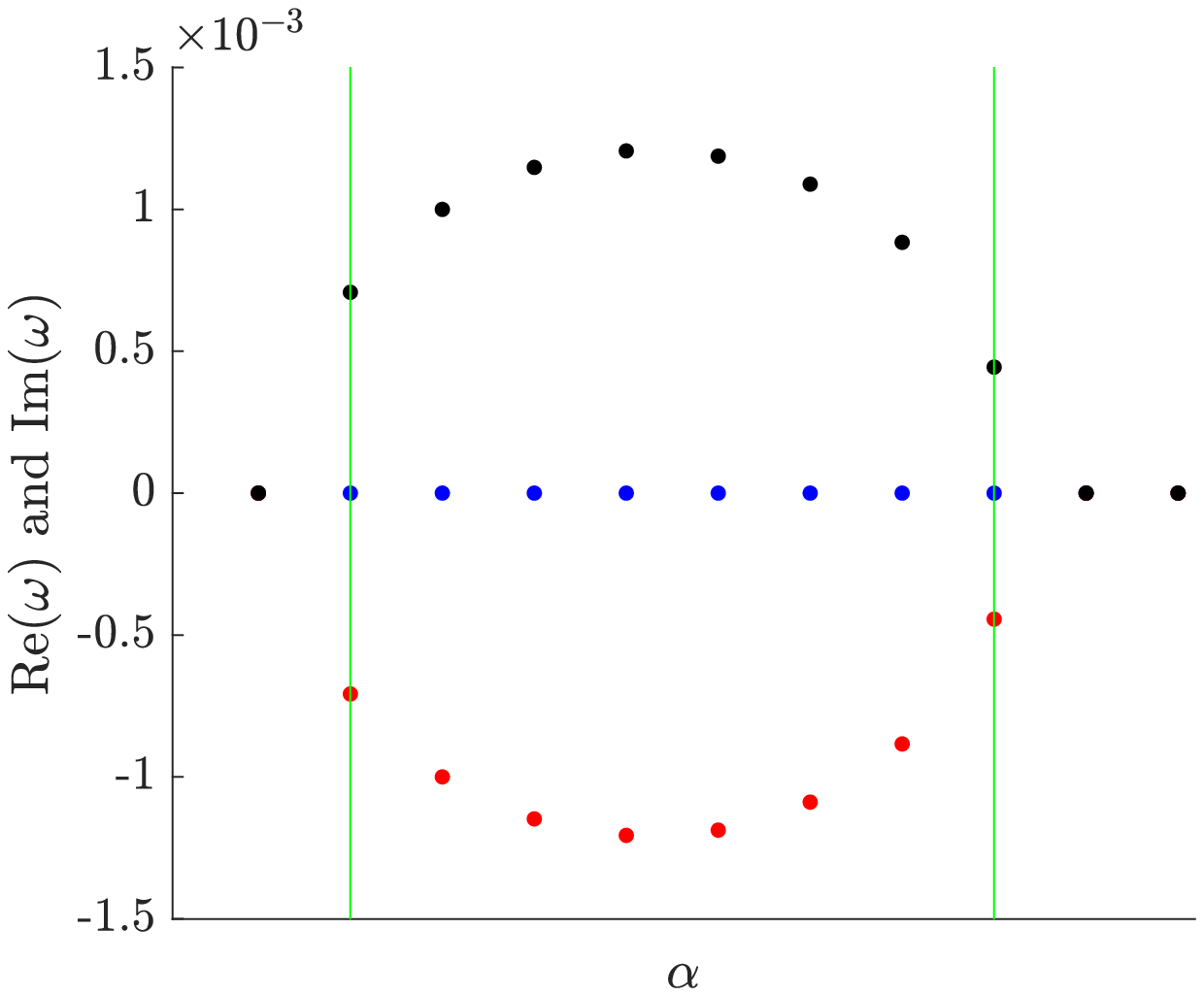}
  \caption{Zoom in of the second k-gap in Figure \ref{fig:kgap}.} 
\end{subfigure}
\caption{Transmission properties of a time modulated one-dimensional lattice of trimers. We set $\Omega=0.35$ and the amplitudes for the respective modulations to be $\varepsilon=0.1$.} 
\label{fig:6}
\end{figure}

\section{Numerical simulations of transmission properties in other structures}
\label{sec4}
In the previous section, we have explained the fundamental reasons for the broken reciprocity and the appearance of k-gaps in the band structure of a chain of trimers.  We have analysed the perturbations of the Floquet exponents emerging from  degenerate points of the static folded chain of trimers asymptotically in the modulation amplitude. In this section, we present numerical examples for other time modulated structures. In \citep{jinghaothesis}, the non-reciprocity property in different structures has been discussed. Therefore,  we focus here on the k-gaps resulting from modulating $\kappa$ only.

\subsection{Square lattice}
We begin by considering resonators in a 2-dimensional square lattice defined through the lattice vectors
		\begin{equation}l_1 = \begin{pmatrix}1\\0\end{pmatrix}, \quad l_2 = \begin{pmatrix}0\\1\end{pmatrix}.\end{equation}
		The lattice and the corresponding Brillouin zone are illustrated in \Cref{fig:square}. The symmetry points in $Y^*$ are given by $\Gamma = (0,0), \ \text{M} = (\pi,\pi)$ and $\text{X}=(\pi,0)$.
\begin{figure}[H]
	\begin{subfigure}[b]{0.5\linewidth}
				\centering
				\begin{tikzpicture}[scale=1.2]
					\begin{scope}[xshift=-5cm,scale=1]
						\pgfmathsetmacro{\r}{0.1pt} 
						\coordinate (a) at (1,0);		
						\coordinate (b) at (0,1);
						
						\draw[opacity=0.2] (0,0) -- (1,0);
						\draw[opacity=0.2] (0,0) -- (0,1);
						\draw[fill=lightgray] (0.76,0.65) circle(\r); 
						\draw[fill=lightgray] (0.24,0.65) circle(\r);
						\draw[fill=lightgray] (0.5,0.2) circle(\r);
						
						\begin{scope}[shift = (a)]					
							\draw[opacity=0.2] (1,0) -- (1,1);
							\draw[opacity=0.2] (0,0) -- (1,0);
							\draw[opacity=0.2] (0,0) -- (0,1);
						\draw[fill=lightgray] (0.76,0.65) circle(\r); 
						\draw[fill=lightgray] (0.24,0.65) circle(\r);
						\draw[fill=lightgray] (0.5,0.2) circle(\r);
						\end{scope}
						\begin{scope}[shift = (b)]
							\draw[opacity=0.2] (1,1) -- (0,1);
							\draw[opacity=0.2] (0,0) -- (1,0);
							\draw[opacity=0.2] (0,0) -- (0,1);
						\draw[fill=lightgray] (0.76,0.65) circle(\r); 
						\draw[fill=lightgray] (0.24,0.65) circle(\r);
						\draw[fill=lightgray] (0.5,0.2) circle(\r);
						\end{scope}
						\begin{scope}[shift = ($-1*(a)$)]
							\draw[opacity=0.2] (0,0) -- (1,0);
							\draw[opacity=0.2] (0,0) -- (0,1);
						\draw[fill=lightgray] (0.76,0.65) circle(\r); 
						\draw[fill=lightgray] (0.24,0.65) circle(\r);
						\draw[fill=lightgray] (0.5,0.2) circle(\r);
						\end{scope}
						\begin{scope}[shift = ($-1*(b)$)]
							\draw[opacity=0.2] (0,0) -- (1,0);
							\draw[opacity=0.2] (0,0) -- (0,1);
						\draw[fill=lightgray] (0.76,0.65) circle(\r); 
						\draw[fill=lightgray] (0.24,0.65) circle(\r);
						\draw[fill=lightgray] (0.5,0.2) circle(\r);
						\end{scope}
						\begin{scope}[shift = ($(a)+(b)$)]
							\draw[opacity=0.2] (1,0) -- (1,1) -- (0,1);
							\draw[opacity=0.2] (0,0) -- (1,0);
							\draw[opacity=0.2] (0,0) -- (0,1);
						\draw[fill=lightgray] (0.76,0.65) circle(\r); 
						\draw[fill=lightgray] (0.24,0.65) circle(\r);
						\draw[fill=lightgray] (0.5,0.2) circle(\r);
						\end{scope}
						\begin{scope}[shift = ($-1*(a)-(b)$)]
							\draw[opacity=0.2] (0,0) -- (1,0);
							\draw[opacity=0.2] (0,0) -- (0,1);
						\draw[fill=lightgray] (0.76,0.65) circle(\r); 
						\draw[fill=lightgray] (0.24,0.65) circle(\r);
						\draw[fill=lightgray] (0.5,0.2) circle(\r);
						\end{scope}
						\begin{scope}[shift = ($(a)-(b)$)]
							\draw[opacity=0.2] (1,0) -- (1,1);
							\draw[opacity=0.2] (0,0) -- (1,0);
							\draw[opacity=0.2] (0,0) -- (0,1);
						\draw[fill=lightgray] (0.76,0.65) circle(\r); 
						\draw[fill=lightgray] (0.24,0.65) circle(\r);
						\draw[fill=lightgray] (0.5,0.2) circle(\r);
						\end{scope}
						\begin{scope}[shift = ($-1*(a)+(b)$)]
							\draw[opacity=0.2] (1,1) -- (0,1);
							\draw[opacity=0.2] (0,0) -- (1,0);
							\draw[opacity=0.2] (0,0) -- (0,1);
						\draw[fill=lightgray] (0.76,0.65) circle(\r); 
						\draw[fill=lightgray] (0.24,0.65) circle(\r);
						\draw[fill=lightgray] (0.5,0.2) circle(\r);
						\end{scope}
						\begin{scope}[shift = ($2*(a)$)]
							\draw (0.5,0.5) node[rotate=0]{$\cdots$};
						\end{scope}
						\begin{scope}[shift = ($-2*(a)$)]
							\draw (0.5,0.5) node[rotate=0]{$\cdots$};
						\end{scope}
						\begin{scope}[shift = ($2*(b)$)]
							\draw (0.5,0.3) node[rotate=90]{$\cdots$};
						\end{scope}
						\begin{scope}[shift = ($-2*(b)$)]
							\draw (0.5,0.7) node[rotate=90]{$\cdots$};
						\end{scope}
					\end{scope}
				\end{tikzpicture}
				\caption{Circular resonators in square lattice.}
			\end{subfigure}
			\begin{subfigure}[b]{0.5\linewidth}
				\centering
				\begin{tikzpicture}[scale=3]	
					\coordinate (a) at ({1/sqrt(3)},1);	
					\coordinate (b) at ({1/sqrt(3)},-1);
					\coordinate (c) at ({2/sqrt(3)},0);
					\coordinate (M) at (0.5,0.5);
					\coordinate (M2) at (-0.5,0.5);
					\coordinate (M3) at (-0.5,-0.5);
					\coordinate (M4) at (0.5,-0.5);
					\coordinate (X) at (0.5,0);
					\coordinate (X2) at (-0.5,0);

					\draw[->,opacity=0.8] (0,0) -- (0.8,0);
					\draw[->,opacity=0.8] (0,0) -- (0,0.8);
					\draw[fill] (M) circle(1pt) node[yshift=8pt, xshift=-2pt]{M}; 	
					\draw[fill] (0,0) circle(1pt) node[left]{$\Gamma$}; 
					\draw[fill] (X) circle(1pt) node[below right]{X}; 

					\draw[thick, postaction={decorate}, decoration={markings, mark=at position 0.1 with {\arrow{>}}, markings, mark=at position 0.3 with {\arrow{>}}, markings,  mark=at position 0.4 with {\arrow{>}}, markings, mark=at position 0.6 with {\arrow{>}}, markings, mark=at position 0.7 with {\arrow{>}}, markings, mark=at position 0.9 with {\arrow{>}}}, color=red]
					(0,0) -- (M) -- (X) -- (0,0);
					\draw[opacity=0.8] (M) -- (M2) -- (M3) -- (M4) -- cycle; 
				\end{tikzpicture}
				\vspace{15pt}
				\caption{Brillouin zone and the symmetry points $\Gamma$, $\mathrm{X}$ and $\mathrm{M}$.}
			\end{subfigure}
			\caption{Illustration of the square lattice and the corresponding Brillouin zone. The red path shows the points where the band functions are computed.} \label{fig:square}
\end{figure}
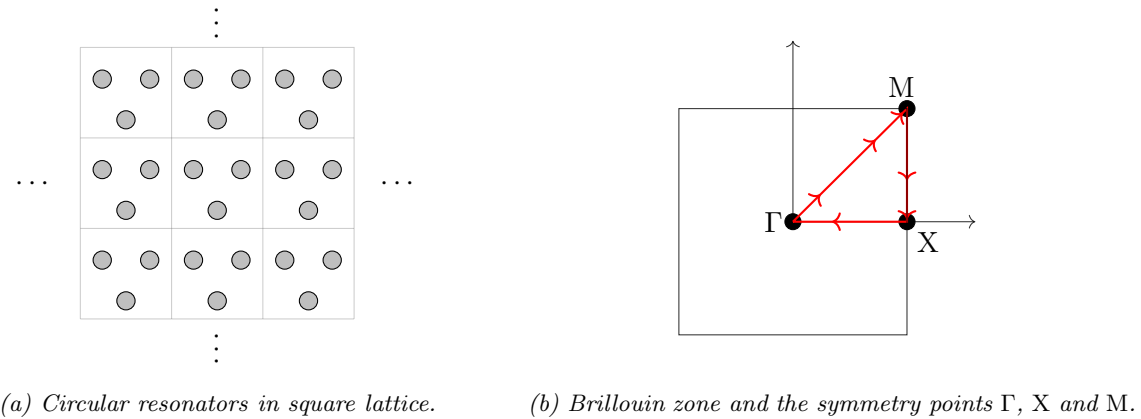
In Figure \ref{fig5}, we compute the band structure with modulation frequency $\Omega = 0.2$ and show the existence of a k-gap. 

\begin{figure}[H]
	\begin{subfigure}[b]{0.48\linewidth}
		\vspace{0pt}
		\begin{center}
			\includegraphics[width=1\linewidth]{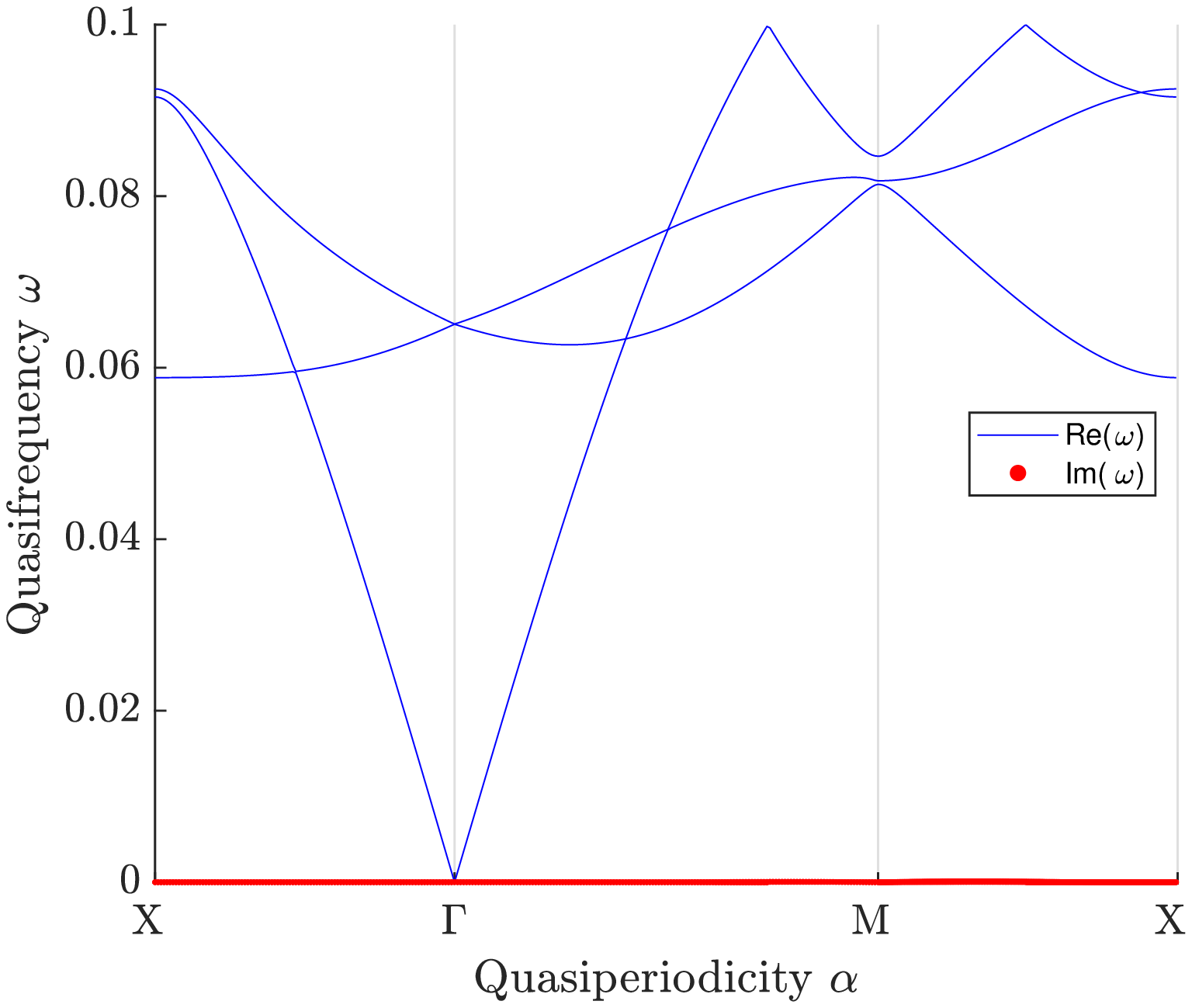}
		\end{center}
		\caption{ \centering
			Static case ($\varepsilon=0$).}
	\end{subfigure}
	\hspace{10pt}
	\begin{subfigure}[b]{0.48\linewidth}
			\vspace{0pt}
		\begin{center}
			\includegraphics[width=1\linewidth]{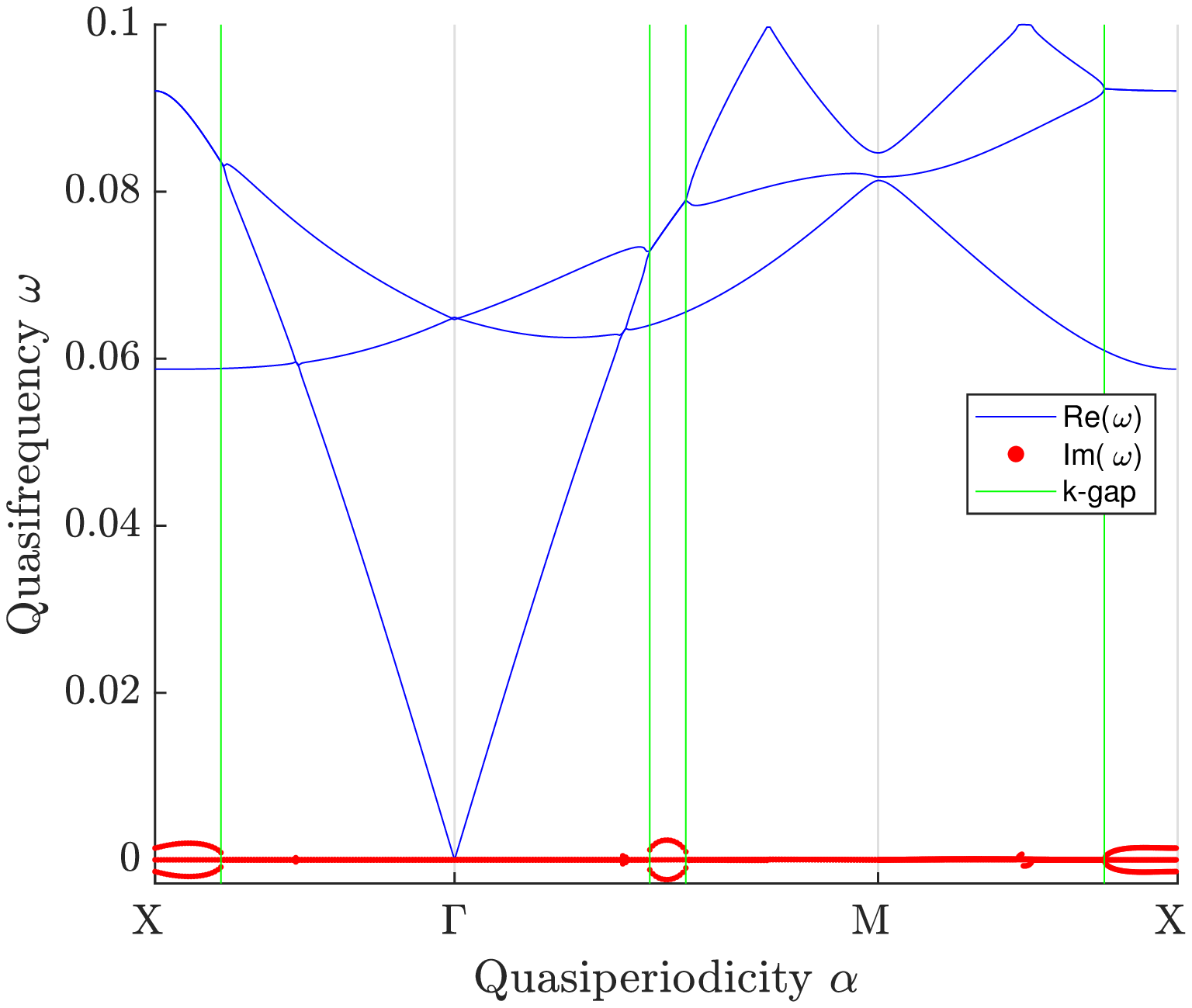}
		\end{center}
		\caption{\centering Modulated case with $\varepsilon=0.1$.}
	\end{subfigure}
	
	\hspace{10pt}
	\caption{Band structure of a square lattice with three subwavelength resonators with $\kappa$ modulation frequency $\Omega = 0.2$.} \label{fig5}
\end{figure}
\subsection{Honeycomb lattice}
First, we consider a honeycomb lattice of resonator trimers as illustrated in Figure \ref{fig:honeycomb}, where the unit cell now contains six resonators $D_i$ respectively centred at $c_i$, $i=1,..,6$:
\begin{align*}
	c_1 &= (1,0) + 3R(1,0), \quad c_2 = (1,0) + 3R\left(\cos\left(\tfrac{2\pi}{3}\right),\sin\left(\tfrac{2\pi}{3}\right)\right),   &c_3 = (1,0) + 3R\left(\cos\left(\tfrac{4\pi}{3}\right),\sin\left(\tfrac{4\pi}{3}\right)\right),\\[0.5em]
	c_4 &= (2,0) + 3R\left(\cos\left(\tfrac{\pi}{3}\right),\sin\left(\tfrac{\pi}{3}\right)\right), \qquad c_5 = (2,0) - 3R(1,0), 	 &c_6 = (2,0) + 3R\left(\cos\left(\tfrac{5\pi}{3}\right),\sin\left(\tfrac{5\pi}{3}\right)\right).
\end{align*}
We use the modulation given by $\kappa_i(t) = 1, \ i=1,\ldots,6$ and 
\begin{equation*}\rho_1(t) = \rho_4(t) = \frac{1}{1 + \varepsilon\cos(\Omega t)}, \quad \rho_2(t) = \rho_5(t) = \frac{1}{1 + \varepsilon\cos\left(\Omega t + \frac{2\pi}{3}\right)}, 
\end{equation*}
\begin{equation*}
\quad \rho_3(t) = \rho_6(t) = \frac{1}{1 + \varepsilon\cos\left(\Omega t + \frac{4\pi}{3}\right)},\end{equation*}
for $0 \leq \varepsilon < 1$.

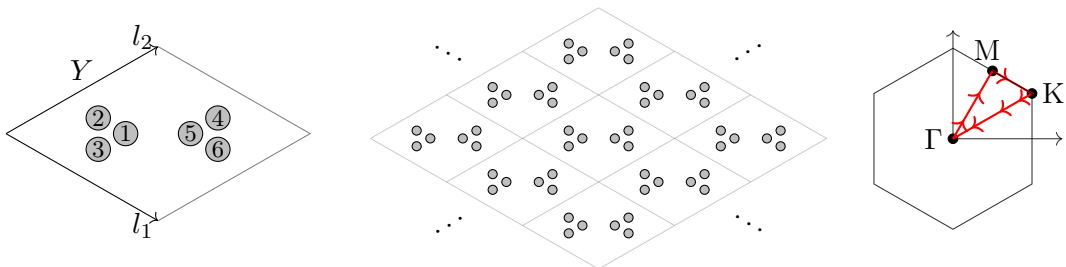
\begin{figure}[H]
	\begin{subfigure}[b]{0.33\linewidth}
		\centering
		\begin{tikzpicture}[scale=2]
				\pgfmathsetmacro{\r}{0.08pt}
				\pgfmathsetmacro{\rt}{0.12pt}
				\coordinate (a) at (1,{1/sqrt(3)});		
				\coordinate (b) at (1,{-1/sqrt(3)});	
				\coordinate (c) at (2,0);
				\coordinate (x1) at ({2/3},0);
				\coordinate (x2) at ({4/3},0);
				\draw[->] (0,0) -- (a) node[pos=0.9,xshift=0,yshift=7]{ $l_2$} node[pos=0.5,above]{$Y$};
				\draw[->] (0,0) -- (b) node[pos=0.9,xshift=0,yshift=-5]{ $l_1$};
				\draw[opacity=0.5] (a) -- (c) -- (b);
				\draw[fill=lightgray] (x1){} +(0:\rt) circle(\r) node{\footnotesize  1};
				\draw[fill=lightgray] (x1){} +(120:\rt) circle(\r)node{\footnotesize  2};
				\draw[fill=lightgray] (x1){} +(240:\rt) circle(\r) node{\footnotesize  3};
				\draw[fill=lightgray] (x2){} +(60:\rt) circle(\r) node{\footnotesize  4};
				\draw[fill=lightgray] (x2){} +(180:\rt) circle(\r) node{\footnotesize  5};
				\draw[fill=lightgray] (x2){} +(300:\rt) circle(\r) node{\footnotesize  6};
		\end{tikzpicture}
		\vspace{0.3cm}
		\caption{Hexagonal lattice unit cell $Y$ containing $6$ resonators.}
	\end{subfigure}
	\begin{subfigure}[b]{0.33\linewidth}
		\begin{tikzpicture}[scale=1]
			\begin{scope}[xshift=-5cm,scale=1]
				\pgfmathsetmacro{\r}{0.06pt}
				\pgfmathsetmacro{\rt}{0.12pt}
				\coordinate (a) at (1,{1/sqrt(3)});		
				\coordinate (b) at (1,{-1/sqrt(3)});
				
				\draw[opacity=0.2] (0,0) -- (a);
				\draw[opacity=0.2] (0,0) -- (b);
				\draw[fill=lightgray] ({2/3},0){} +(0:\rt) circle(\r);
				\draw[fill=lightgray] ({2/3},0){} +(120:\rt) circle(\r);
				\draw[fill=lightgray] ({2/3},0){} +(240:\rt) circle(\r);
				\draw[fill=lightgray] ({4/3},0){} +(60:\rt) circle(\r);
				\draw[fill=lightgray] ({4/3},0){} +(180:\rt) circle(\r);
				\draw[fill=lightgray] ({4/3},0){} +(300:\rt) circle(\r);
				
				\begin{scope}[shift = (a)]					
					\draw[opacity=0.2] (0,0) -- (1,{1/sqrt(3)});
					\draw[opacity=0.2] (0,0) -- (1,{-1/sqrt(3)});
					\draw[opacity=0.2] (1,{1/sqrt(3)}) -- (2,0);
					\draw[fill=lightgray] ({2/3},0){} +(0:\rt) circle(\r);
					\draw[fill=lightgray] ({2/3},0){} +(120:\rt) circle(\r);
					\draw[fill=lightgray] ({2/3},0){} +(240:\rt) circle(\r);
					\draw[fill=lightgray] ({4/3},0){} +(60:\rt) circle(\r);
					\draw[fill=lightgray] ({4/3},0){} +(180:\rt) circle(\r);
					\draw[fill=lightgray] ({4/3},0){} +(300:\rt) circle(\r);
				\end{scope}
				\begin{scope}[shift = (b)]
					
					\draw[opacity=0.2] (0,0) -- (1,{1/sqrt(3)});
					\draw[opacity=0.2] (0,0) -- (1,{-1/sqrt(3)});
					\draw[opacity=0.2] (2,0) -- (1,{-1/sqrt(3)});
					\draw[fill=lightgray] ({2/3},0){} +(0:\rt) circle(\r);
					\draw[fill=lightgray] ({2/3},0){} +(120:\rt) circle(\r);
					\draw[fill=lightgray] ({2/3},0){} +(240:\rt) circle(\r);
					\draw[fill=lightgray] ({4/3},0){} +(60:\rt) circle(\r);
					\draw[fill=lightgray] ({4/3},0){} +(180:\rt) circle(\r);
					\draw[fill=lightgray] ({4/3},0){} +(300:\rt) circle(\r);
				\end{scope}
				\begin{scope}[shift = ($-1*(a)$)]
					
					\draw[opacity=0.2] (0,0) -- (1,{1/sqrt(3)});
					\draw[opacity=0.2] (0,0) -- (1,{-1/sqrt(3)});
					\draw[fill=lightgray] ({2/3},0){} +(0:\rt) circle(\r);
					\draw[fill=lightgray] ({2/3},0){} +(120:\rt) circle(\r);
					\draw[fill=lightgray] ({2/3},0){} +(240:\rt) circle(\r);
					\draw[fill=lightgray] ({4/3},0){} +(60:\rt) circle(\r);
					\draw[fill=lightgray] ({4/3},0){} +(180:\rt) circle(\r);
					\draw[fill=lightgray] ({4/3},0){} +(300:\rt) circle(\r);
				\end{scope}
				\begin{scope}[shift = ($-1*(b)$)]
					
					\draw[opacity=0.2] (0,0) -- (1,{1/sqrt(3)});
					\draw[opacity=0.2] (0,0) -- (1,{-1/sqrt(3)});
					\draw[fill=lightgray] ({2/3},0){} +(0:\rt) circle(\r);
					\draw[fill=lightgray] ({2/3},0){} +(120:\rt) circle(\r);
					\draw[fill=lightgray] ({2/3},0){} +(240:\rt) circle(\r);
					\draw[fill=lightgray] ({4/3},0){} +(60:\rt) circle(\r);
					\draw[fill=lightgray] ({4/3},0){} +(180:\rt) circle(\r);
					\draw[fill=lightgray] ({4/3},0){} +(300:\rt) circle(\r);
				\end{scope}
				\begin{scope}[shift = ($(a)+(b)$)]
					
					\draw[opacity=0.2] (0,0) -- (1,{1/sqrt(3)});
					\draw[opacity=0.2] (0,0) -- (1,{-1/sqrt(3)});
					\draw[opacity=0.2] (1,{1/sqrt(3)}) -- (2,0) -- (1,{-1/sqrt(3)});
					\draw[fill=lightgray] ({2/3},0){} +(0:\rt) circle(\r);
					\draw[fill=lightgray] ({2/3},0){} +(120:\rt) circle(\r);
					\draw[fill=lightgray] ({2/3},0){} +(240:\rt) circle(\r);
					\draw[fill=lightgray] ({4/3},0){} +(60:\rt) circle(\r);
					\draw[fill=lightgray] ({4/3},0){} +(180:\rt) circle(\r);
					\draw[fill=lightgray] ({4/3},0){} +(300:\rt) circle(\r);
				\end{scope}
				\begin{scope}[shift = ($-1*(a)-(b)$)]
					
					\draw[opacity=0.2] (0,0) -- (1,{1/sqrt(3)});
					\draw[opacity=0.2] (0,0) -- (1,{-1/sqrt(3)});
					\draw[fill=lightgray] ({2/3},0){} +(0:\rt) circle(\r);
					\draw[fill=lightgray] ({2/3},0){} +(120:\rt) circle(\r);
					\draw[fill=lightgray] ({2/3},0){} +(240:\rt) circle(\r);
					\draw[fill=lightgray] ({4/3},0){} +(60:\rt) circle(\r);
					\draw[fill=lightgray] ({4/3},0){} +(180:\rt) circle(\r);
					\draw[fill=lightgray] ({4/3},0){} +(300:\rt) circle(\r);
				\end{scope}
				\begin{scope}[shift = ($(a)-(b)$)]
					
					\draw[opacity=0.2] (0,0) -- (1,{1/sqrt(3)});
					\draw[opacity=0.2] (0,0) -- (1,{-1/sqrt(3)});
					\draw[opacity=0.2] (1,{1/sqrt(3)}) -- (2,0);
					\draw[fill=lightgray] ({2/3},0){} +(0:\rt) circle(\r);
					\draw[fill=lightgray] ({2/3},0){} +(120:\rt) circle(\r);
					\draw[fill=lightgray] ({2/3},0){} +(240:\rt) circle(\r);
					\draw[fill=lightgray] ({4/3},0){} +(60:\rt) circle(\r);
					\draw[fill=lightgray] ({4/3},0){} +(180:\rt) circle(\r);
					\draw[fill=lightgray] ({4/3},0){} +(300:\rt) circle(\r);
				\end{scope}
				\begin{scope}[shift = ($-1*(a)+(b)$)]					
					\draw[opacity=0.2] (0,0) -- (1,{1/sqrt(3)});
					\draw[opacity=0.2] (0,0) -- (1,{-1/sqrt(3)});
					\draw[opacity=0.2] (2,0) -- (1,{-1/sqrt(3)});
					\draw[fill=lightgray] ({2/3},0){} +(0:\rt) circle(\r);
					\draw[fill=lightgray] ({2/3},0){} +(120:\rt) circle(\r);
					\draw[fill=lightgray] ({2/3},0){} +(240:\rt) circle(\r);
					\draw[fill=lightgray] ({4/3},0){} +(60:\rt) circle(\r);
					\draw[fill=lightgray] ({4/3},0){} +(180:\rt) circle(\r);
					\draw[fill=lightgray] ({4/3},0){} +(300:\rt) circle(\r);
				\end{scope}
				\begin{scope}[shift = ($2*(a)$)]
					\draw (1,0) node[rotate=30]{$\cdots$};
				\end{scope}
				\begin{scope}[shift = ($-2*(a)$)]
					\draw (1,0) node[rotate=210]{$\cdots$};
				\end{scope}
				\begin{scope}[shift = ($2*(b)$)]
					\draw (1,0) node[rotate=-30]{$\cdots$};
				\end{scope}
				\begin{scope}[shift = ($-2*(b)$)]
					\draw (1,0) node[rotate=150]{$\cdots$};
				\end{scope}
			\end{scope}
		\end{tikzpicture}
		
		\caption{Periodic system with trimers in a honeycomb lattice.}
	\end{subfigure}
	\begin{subfigure}[b]{0.33\linewidth}
		\centering
			\begin{tikzpicture}[scale=1.8]	
				\coordinate (a) at ({1/sqrt(3)},1);		
				\coordinate (b) at ({1/sqrt(3)},-1);
				\coordinate (c) at ({2/sqrt(3)},0);
				\coordinate (M) at ({0.5/sqrt(3)},0.5);
				\coordinate (M2) at ({-0.5/sqrt(3)},-0.5);
				\coordinate (Km) at ({-1/sqrt(3)},{-1/3});
				\coordinate (K1) at ({1/sqrt(3)},{1/3});
				\coordinate (K2) at ({1/sqrt(3)},{-1/3});
				\coordinate (K3) at (0,{-2/3});
				\coordinate (K4) at ({-1/sqrt(3)},{-1/3});
				\coordinate (K5) at ({-1/sqrt(3)},{1/3});
				\coordinate (K6) at (0,{2/3});
				
				\draw[->,opacity=0.8] (0,0) -- (0.8,0);
				\draw[->,opacity=0.8] (0,0) -- (0,0.8);
				\draw[fill] (M) circle(1pt) node[yshift=8pt, xshift=-2pt]{M}; 
				\draw[fill] (0,0) circle(1pt) node[left]{$\Gamma$}; 
				\draw[fill] (K1) circle(1pt) node[right]{K};
				\draw[thick,postaction={decorate}, decoration={markings, 
				 mark=at position 0.1 with {\arrow{>}}, markings,mark=at position 0.25 with {\arrow{>}},markings, mark=at position 0.45 with {\arrow{>}},markings,  mark=at position 0.65 with {\arrow{>}},markings, mark=at position 0.75 with {\arrow{>}}, markings, mark=at position 0.9 with {\arrow{>}}}, color=red]
				(0,0) -- (M) -- (K1) -- (0,0);
	
				\draw[opacity=0.8] (K1) -- (K2) -- (K3) -- (K4) -- (K5) -- (K6) -- cycle; 
			\end{tikzpicture}
		\vspace{15pt}		
		\caption{Brillouin zone and the symmetry points $\Gamma$, $\mathrm{K}$ and $\mathrm{M}$.}
	\end{subfigure}
	\caption{Illustration of the honeycomb lattice and corresponding Brillouin zone. The red path shows the points where the band functions are computed.} \label{fig:honeycomb}
\end{figure}

\begin{figure}[tbh]

	\begin{subfigure}[b]{0.45\linewidth}
			\vspace{0pt}
		\begin{center}
\includegraphics[width=1\linewidth]{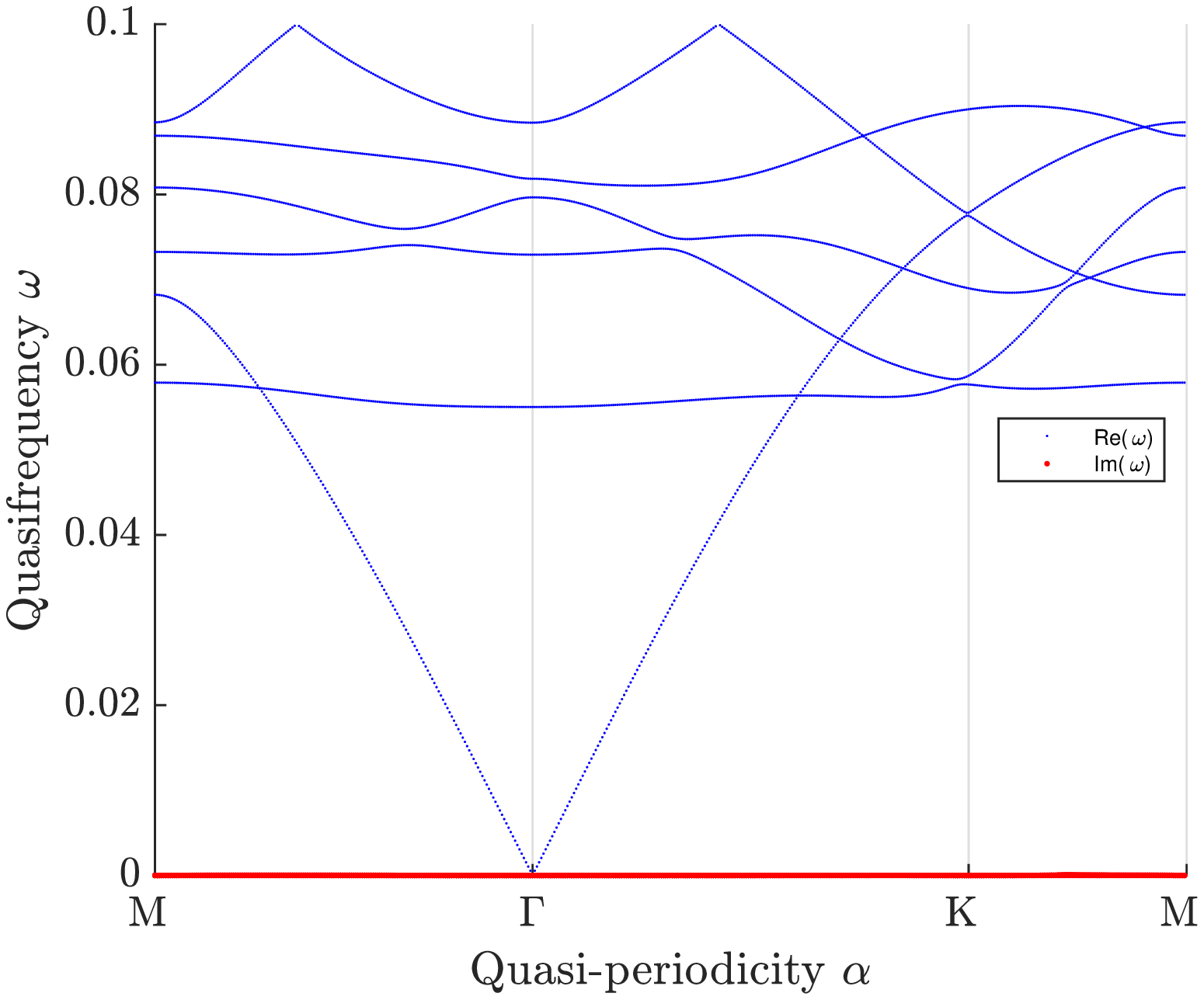}
		\end{center}
		\caption{\centering Static case.}
	\end{subfigure}
	\hspace{0.5cm}
	\centering
	\begin{subfigure}[b]{0.45\linewidth}
			\vspace{0pt}
		\begin{center}
 \includegraphics[width=1\linewidth]{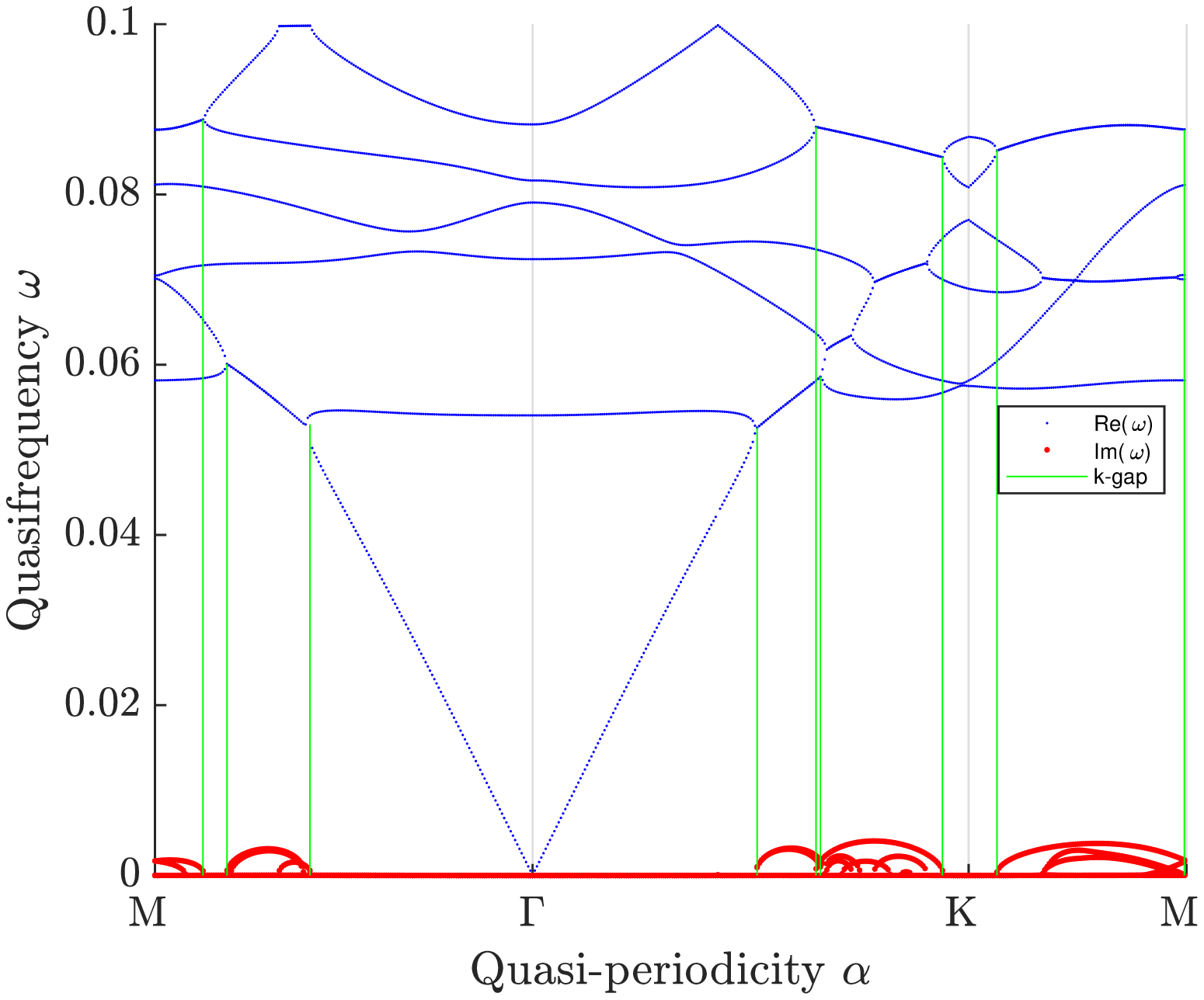}
		\end{center}
		\caption{\centering Modulated case with $\varepsilon=0.2$.}
	\end{subfigure}
	\hspace{10pt}
	\caption{Band structure of a honeycomb lattice with six subwavelength resonators with $\kappa$ modulation frequency $\Omega = 0.2$.} \label{fig7}
\end{figure}

The band functions are given in Figure \ref{fig7}, which illustrates the existence of k-gaps in a honeycomb crystal of subwavelength resonators. We observe in all structures above that the intervals of degeneracy of the band functions correspond to the k-gaps. 

\appendix
\section{Eigenvalue perturbation theory}
Consider (\ref{eq:1dsys}) and assume that $F=F_0+\varepsilon F_1+\varepsilon^2 F_2 + O(\varepsilon^3 )$, and $F_0$ is diagonal with respect to the basis vectors $w_1,\ldots,w_n$.  We would like to expand the eigenvalues of $F$ in terms of $\varepsilon$. As said before, this is a typical problem in perturbative quantum theory \cite{perturbation}. Similar formulas in quantum mechanical perturbation theory can be found in textbooks such as \cite{qmimperial}. The following derivation is reformulated to suit our setting.

We will focus on the perturbation of degenerate points. Let $f_0$ be a degenerate eigenvalue of $F_0$ of multiplicity $r$ and let  $w_1,\ldots, w_r$ be its associated eigenvectors. Without loss of generality, we assume that $(F_0)_{ii} = f_0$ for $i=1,\ldots,r$. We define the projection operator
\begin{equation}
	P:=\begin{pmatrix}
		\mathrm{Id}_r & \\ & 0
	\end{pmatrix}  \quad \text{and let } \ Q:=\mathrm{Id}_N -P.
\end{equation}
Here, $\mathrm{Id}_r$ is the $r\times r$ identity matrix. 

We remark that $F_0$ commutes with $P$ and $Q$. Now, we fix an eigenvector $v_0\in\text{span}\{w_1,\ldots,w_n\}$ and expand $v$ and $f$ as follows:
\begin{equation}
	\begin{split}
		v&=v_0+\varepsilon v_1+\varepsilon^2 v_2+O(\varepsilon^3),\\
		f&=f_0+\varepsilon f_1+\varepsilon^2 f_2+O(\varepsilon^3).
	\end{split}
\end{equation}
We require first that $v_0 = P(v)$, due to the normalization of $v$. From $Fv=fv$,  it follows that up to $O(\varepsilon^2)$ 
\begin{equation}
	\label{Qequation}
	\begin{split}
		& F_0v+\varepsilon(F_1+\varepsilon F_2)v=fv,\\
		& QF_0v+\varepsilon QVv=fQv,\\
		& Q(f \mathrm{Id}_N -F_0)v=\varepsilon QVv \text{ and } \\
		& Qv =\varepsilon ((f \, \mathrm{Id}_N -F_0)^{-1}Q)Vv,\end{split}	
\end{equation}
where $V:= F_1 + \varepsilon F_2$.

Note that we should treat $((f \, \mathrm{Id}_N -F_0)^{-1}Q)$ as $0|_{E_{f_0}}\oplus ((f \, \mathrm{Id}_N -F_0)^{-1}Q)|_{E_{f_0}^c}$, where ${E_{f_0}}$ denotes the eigenspace associated with $f_0$ and $E_{f_0}^c$ its complementary. We get
\begin{equation}
	PF_0v+\varepsilon PVv = fPv,
\end{equation}
and therefore,
\begin{equation}
	\label{Pequation}
	f_0Pv+\varepsilon PVv = fPv, 
\end{equation}
where we used the fact that $PF_0v=F_0Pv=f_0Pv$. Now, inserting $v=Pv+Qv$ into the second term of the left-hand side of (\ref{Pequation}) yields $f_0Pv+\varepsilon PV(Pv+Qv)=fPv$ and hence
\begin{equation}
	\label{master}
	\begin{split}
		& f_0Pv+\varepsilon PVPv+\varepsilon PVQv=fPv\text{ and }\\
		& f_0Pv+ \varepsilon PVPv+ \varepsilon^2PV\left((f \, \mathrm{Id}_N -F_0)^{-1}Q\right)Vv=fPv.
	\end{split}
\end{equation}
For the $\varepsilon^2$-term, we evaluate the expression at $\varepsilon=0$:
\begin{equation}
	\label{hocus}
	PV((f\, \mathrm{Id}_N -F_0)^{-1}Q)Vv|_{\varepsilon=0}=PF_1\left((f_0 \, \mathrm{Id}_N -F_0)^{-1}Q\right)F_1v_0:=PF_1GF_1v_0,
\end{equation}
where $G:=(f_0 \, \mathrm{Id}_N -F_0)^{-1}Q=\text{diag}(0,..,0,(f_0-\lambda_2)^{-1},\ldots,(f_0-\lambda_k)^{-1})$ if we assume that $F_0=\text{diag}(f_0,\ldots,f_0,\lambda_2,\ldots,\lambda_k)$. Hence, we can write that
\begin{equation}
	\label{effectiveH}	
	\underbrace{P\left(F_0 \, \mathrm{Id}_N +\varepsilon (F_1+\varepsilon F_2)+\varepsilon^2(PF_1GF_1)\right)P}_{:=\mathcal{H}}v_0=fv_0.
\end{equation}
We define the \textit{effective Hamiltonian} as
\begin{equation}
	\label{effective}
	\mathcal{H}:=PF_0P+\varepsilon PF_1P +\varepsilon^2P(F_1GF_1+F_2)P.
\end{equation}
\noindent
Hence if $f_0$ is a degenerate point of multiplicity $r$, there will be $r$ corresponding values of $f_1$ which are the non-zero eigenvalues of the matrix $PF_1P$.

\bibliographystyle{abbrv}
\bibliography{paper_kappa}{}
\end{document}